\documentclass[reqno,11pt,twoside]{article}
\usepackage[hmargin=3cm, vmargin=1.in, marginparwidth=2.6cm, marginparsep=0.3cm, a4paper, centering]{geometry}
\usepackage[T1]{fontenc}
\usepackage[lining]{ebgaramond}
\usepackage{amsmath,amsthm}
\usepackage[ebgaramond]{newtxmath}
\usepackage[scr=rsfs]{mathalpha}
\usepackage{enumerate}
\usepackage[font=footnotesize, labelfont=bf, labelsep=colon, margin=0.5in]{caption}
\usepackage{graphicx}

\usepackage{xcolor}
\renewcommand\footnotemark{}
\usepackage[nobottomtitles,pagestyles]{titlesec}
\titleformat{\section}
{\normalfont\large\bfseries}
{\filcenter\S\thesection.}{1ex}{\filcenter}
\titleformat{\subsection}
{\normalfont\bfseries}
{\filcenter\S\thesubsection.}{1ex}{\filcenter}
\usepackage[colorlinks,allcolors=blue,pagebackref]{hyperref}
\usepackage{pifont}
\renewcommand*{\backrefalt}[4]{%
\ifcase #1 %
\textcolor{red}{No citations}%
\or
\ding{43}~p.~#2%
\else
\ding{43}~pp.~#2%
\fi}
\usepackage{breakcites}
\newcommand{\N}{\mathcal N}

\newcommand{\R}{{\mathbb R}}

\DeclareMathOperator*{\Spec}{Spec}
\newcommand{\DtN}{{\mathcal D}}
\newcommand{\myscal}[1]{\left(#1\right)}

\newcommand{\Dom}{\mathrm{Dom}}
\newcommand{\Dir}{\mathrm{Dir}}
\newcommand{\Neu}{\mathrm{Neu}}
\newcommand{\Rob}{\mathrm{Rob}}

\newcommand{\dr}{\mathrm{d}}
\newcommand{\er}{\mathrm{e}}
\newcommand{\ir}{\mathrm{i}}

\renewcommand{\tilde}{\widetilde}
\newcommand{\dist}{\operatorname{dist}}
\numberwithin{equation}{section}
\usepackage{tcolorbox,varwidth}
\tcbuselibrary{skins, breakable, theorems}
\colorlet{thmcolor}{orange!80!white}
\colorlet{remcolor}{teal!50!white}
\colorlet{defncolor}{blue!50!black}
\usepackage{keytheorems}
\tcbset{commonstyle/.style={%
colback=white!95!tcbcolframe,
colbacktitle=white!80!tcbcolframe,
arc=2mm,
fonttitle=\bfseries,
coltitle=black,
enhanced,
varwidth boxed title*=-2cm,
attach boxed title to top left={yshift=-3mm,xshift=0.5cm,yshifttext=-1mm},
description font=\mdseries,
breakable,
before upper={\parindent15pt\noindent},
beforeafter skip balanced=10.0pt plus 1.0pt minus 1.0pt,
separator sign={\textbf{: }},
overlay first={\draw[color=tcbcolframe, line width=.5pt] (frame.south west)--(frame.south east);},
overlay middle={\draw[color=tcbcolframe, line width=.5pt] (frame.south west)--(frame.south east);\draw[color=tcbcolframe, line width=.5pt] (frame.north west)--(frame.north east);},
overlay last={\draw[color=tcbcolframe, line width=.5pt] (frame.north west)--(frame.north east);}
}
}
\newkeytheoremstyle{mythmstyle}{
bodyfont=\normalfont,
headpunct={},
notebraces={}{},
noteseparator={: }, 
notefont=\bfseries,
tcolorbox = {colframe=thmcolor,commonstyle}
}
\newkeytheoremstyle{myremstyle}{
bodyfont=\normalfont,
headpunct={},
notebraces={}{},
noteseparator={: }, 
notefont=\bfseries,
tcolorbox = {colframe=remcolor,commonstyle}
}
\newkeytheoremstyle{mydefstyle}{
bodyfont=\normalfont,
headpunct={},
notebraces={}{},
noteseparator={: }, 
notefont=\bfseries,
tcolorbox = {colframe=defncolor,commonstyle}
}
\newkeytheorem{theorem}[
name=Theorem,
style=mythmstyle,
parent=section
]
\newkeytheorem{lemma}[
name=Lemma,
style=mythmstyle,
sibling=theorem
]
\newkeytheorem{prop}[
name=Proposition,
style=mythmstyle,
sibling=theorem
]
\newkeytheorem{corollary}[
name=Corollary,
style=mythmstyle,
sibling=theorem
]
\newkeytheorem{remark}[
name=Remark,
sibling=theorem,
style=myremstyle
]
\newkeytheorem{definition}[
name=Definition,
style=mydefstyle,
sibling=theorem
]
\newkeytheorem{namedconj}[
name=Conjecture,
style=mythmstyle
]
 
\let\originalleft\left
\let\originalright\right
\renewcommand{\left}{\mathopen{}\mathclose\bgroup\originalleft}
\renewcommand{\right}{\aftergroup\egroup\originalright}
\newcommand{\mydoi}[1]{\href{https://doi.org/#1}{doi: #1}}
\newcommand{\myarXiv}[1]{\href{https://arxiv.org/abs/#1}{arXiv: #1}}
\title{Comparison inequalities for Dirichlet-to-Neumann maps%
}
\date{\footnotesize arXiv v.1; 19 August 2026}
\author{
Denis S. Grebenkov\thanks{\textbf{D. S. G.}: Laboratoire de Physique de la Mati\`ere Condens\'ee, CNRS -- \'Ecole Polytechnique, Institut Polytechnique de Paris, 91120 Palaiseau, France; \href{mailto:denis.grebenkov@polytechnique.edu}{denis.grebenkov@polytechnique.edu}; \url{https://pmc.polytechnique.fr/pagesperso/dg/}; ORCID: 0000-0002-6273-9164%
}
\and 
Michael Levitin\thanks{%
\textbf{M. L.}: Department of Mathematics and Statistics, University of Reading, 
Pepper Lane, Whiteknights, Reading RG6 6AX, UK;
\href{mailto:M.Levitin@reading.ac.uk}{M.Levitin@reading.ac.uk}; \url{https://www.michaellevitin.net}; ORCID: 0000-0003-0020-3265%
}
\and
Karl-Mikael Perfekt\thanks{%
\textbf{K.-M. P.}:  Department of Mathematical Sciences, Norwegian University of Science and Technology (NTNU), 7491 Trondheim, Norway;
\href{mailto:karl-mikael.perfekt@ntnu.no}{\nolinkurl{karl-mikael.perfekt@ntnu.no}}; \url{https://kmperfekt.com}; ORCID: 0000-0003-4574-9683%
}
\and
Iosif Polterovich\thanks{%
\textbf{I. P.}: D\'e\-par\-te\-ment de math\'ematiques et de statistique, Univer\-sit\'e de Mont\-r\'eal, 
CP 6128 succ Centre-Ville, Mont\-r\'eal QC  H3C 3J7, Canada;
\href{mailto:iosif.polterovich@umontreal.ca}{\nolinkurl{iosif.polterovich@umontreal.ca}}; \url{https://www.dms.umontreal.ca/\~iossif}; ORCID: 0009-0007-0052-6589%
}
}
\begin{document}
\maketitle
\begin{abstract}
We prove comparison inequalities for Dirichlet-to-Neumann maps corresponding to different non-positive Helmholtz parameters. For convex domains our bounds are sharp, and the resulting eigenvalue inequalities partially confirm an earlier conjecture, which we show does not hold in full generality. We further obtain geometry-dependent versions for arbitrary sufficiently regular domains, together with extensions to compact Riemannian manifolds with boundary. We also discuss analogous questions for metric graphs.
\end{abstract}
{\small \tableofcontents}
\section{Introduction}\label{sec:intro}
\subsection{Statement of the problems}
Let $\Omega\subset\mathbb{R}^d$, $d\ge 2$, be a bounded domain with Lipschitz boundary. For $\Lambda\in\mathbb{R}\setminus\Spec\left(-\Delta^\Dir_\Omega\right)$ and $u\in H^{1/2}(\partial\Omega)$, let $U=\mathcal{E}_\Lambda u\in H^1(\Omega)$ be a $\Lambda$-harmonic extension of $u$, that is, the unique  weak solution of 
\[
\begin{cases}
-\Delta U - \Lambda U=0\qquad&\text{in }\Omega,\\
U|_{\partial\Omega} = u.\qquad&
\end{cases}
\]
We define the \emph{Dirichlet-to-Neumann map} $\DtN_\Lambda:H^{1/2}(\partial\Omega)\to H^{-1/2}(\partial\Omega)$ as $\DtN_\Lambda: u\mapsto \partial_n \mathcal{E}_\Lambda u$, where $\partial_n$ is the normal derivative on the boundary, understood in the sense of the  form
\begin{equation}\label{eq:quadDtN}
\begin{gathered}
\mathfrak{d}_\Lambda[u,v]:=\myscal{\DtN_\Lambda u, v}_{L^2(\partial\Omega)}=\myscal{\nabla \mathcal{E}_\Lambda u, \nabla \mathcal{E}_\Lambda v}_{L^2(\Omega)}-\Lambda \myscal{\mathcal{E}_\Lambda u, \mathcal{E}_\Lambda v}_{L^2(\Omega)},\\
\mathfrak{d}_\Lambda[u]:= \mathfrak{d}_\Lambda[u,u], \qquad
u, v\in H^{1/2}(\partial\Omega).
\end{gathered}
\end{equation}
The Dirichlet-to-Neumann map $\DtN_\Lambda$ is  a self-adjoint operator in $L^2(\partial\Omega)$ with a discrete spectrum; we denote its eigenvalues, enumerated in non-decreasing order with account of multiplicities,
\[
\sigma_1^{(\Lambda)}\le \sigma_2^{(\Lambda)}\le \dots \to +\infty.
\]
By restricting its domain to functions orthogonal to the normal derivatives of the elements of the corresponding Dirichlet Laplacian eigenspace, the definition of $\DtN_\Lambda$ is naturally extended to $\Lambda\in\Spec\left(-\Delta^\Dir_\Omega\right)$.

If $\Lambda$ is below the first Dirichlet Laplacian eigenvalue of $\Omega$, $\Lambda<\lambda_1^\Dir(\Omega)$ (and, in particular, if $\Lambda\le 0$), then for any $u\in H^{1/2}(\partial\Omega)$
\begin{equation}\label{eq:DirPr}
\mathfrak{d}_\Lambda[u] = \min_{\substack{W\in H^1(\Omega)\\W|_{\partial\Omega}=u}} \left(\|\nabla W\|^2_{L^2(\Omega)}-\Lambda\|W\|^2_{L^2(\Omega)}\right),
\end{equation}
with the minimum attained at $W=\mathcal{E}_\Lambda u$, cf. \cite[Proposition 2.5]{GLP26}. Hence, the eigenvalues
$\sigma_k^{(\Lambda)}$ obey the variational principle
\begin{equation} \label{eq:var_pr}
\sigma_k^{(\Lambda)} = \min_{\substack{\mathcal{L} \subset H^1(\Omega)\\\dim\mathcal{L} = k}}\  \max_{W \in \mathcal{L}\setminus\{0\}} 
\frac{\|\nabla W\|^2_{L^2(\Omega)} - \Lambda \|W\|^2_{L^2(\Omega)}}{\|W|_{\partial\Omega}\|^2_{L^2(\partial\Omega)}},  \qquad k\in\mathbb{N}. 
\end{equation}
For $\Lambda>\lambda_1^\Dir(\Omega)$, the space $H^1(\Omega)$ in the variational principle \eqref{eq:var_pr} should be replaced by the subspace of all $\Lambda$-harmonic functions from $H^1(\Omega)$.

For $\Lambda\le 0$, the operator $\DtN_\Lambda$ is non-negative, with $\sigma_1^{(0)}=0$ and $\sigma_1^{(\Lambda)}>0$ for $\Lambda<0$. 

On any interval of the real line non-containing the eigenvalues of the Dirichlet Laplacian $-\Delta_\Omega^\Dir$, each eigenvalue $\sigma_k^{(\Lambda)}$, $k\in \mathbb{N}$, is a strictly monotone decreasing function of $\Lambda$. Moreover, the union of the spectra of $\DtN_\Lambda$ on any such interval can be decomposed, at the cost of forfeiting  the ordering of eigenvalues, into a family of real-analytic monotone decreasing curves $\sigma^{(\Lambda)}$. 
For further details we refer to \cite{Arendt12, Behrndt15, GLP26}.

The following conjecture about the eigenvalues of $\DtN_\Lambda$ for $\Lambda\le 0$ was made in the first  arXiv version of \cite{GLP26}.

\begin{namedconj}\label{conj:A}
For any bounded Lipschitz domain $\Omega\subset\mathbb{R}^d$, any $\Lambda\le  0$, and any fixed  real-analytic branch $\sigma^{(\Lambda)}$ 
\[
\sigma^{(\Lambda)}-\sigma^{(0)}\le \sqrt{-\Lambda}.
\]
\end{namedconj}

A weaker variant of Conjecture \ref{conj:A} is 

\begin{namedconj}\label{conj:B}
For any bounded Lipschitz domain $\Omega\subset\mathbb{R}^d$, any $\Lambda\le  0$, and any $k\in\mathbb{N}$, 
\begin{equation}\label{eq:conjB}
\sigma_k^{(\Lambda)}-\sigma_k^{(0)}\le \sqrt{-\Lambda}.
\end{equation}
\end{namedconj}

\begin{remark} Conjecture \ref{conj:B} can be restated in terms of the eigenvalue counting function (see  Definition \ref{defn:counting}) of $\DtN_\Lambda$ as the fact that 
\[
\mathcal{N}^{\DtN_\Lambda}\left(\sqrt{-\Lambda}+\sigma_k^{(0)}\right) \ge k, 
\]
holds for all $k\in\mathbb{N}$ and $\Lambda\le 0$.
\end{remark}

There are various sources of motivation for Conjectures \ref{conj:A} and \ref{conj:B}. They are known to be true, in full generality, for disks and balls \cite{GrCh}.  The validity of Conjecture \ref{conj:B} for the principal eigenvalue $\sigma_1^{(\Lambda)}$ is well-known for an arbitrary domain, see  \cite[Theorem 4.17]{GLP26} and related results \cite[Theorem 2.3]{GioSmi}, \cite[Lemma 2.1]{DanersKennedy}, and \cite[Proposition 4.12]{Bucur17}.
Conjectures \ref{conj:A} and \ref{conj:B} have been confirmed by some (limited) numerical experiments, see Figure \ref{fig:rect}. Another uniform bound hinting at a possible $\sqrt{-\Lambda}$ behaviour of eigenvalues of the Dirichlet-to-Neumann map has been proved in \cite{Girouard22}: if a domain $\Omega \subset\R^d$ has a sufficiently smooth boundary, then, 
with some constant $C > 0$ depending only on $\Omega$, and with $0=\nu_1<\nu_2\le\dots$, denoting the eigenvalues of the Laplace--Beltrami operator $-\Delta_{\partial\Omega}$, the
bounds
\[
\left|\sigma_k^{(\Lambda)} - \sqrt{- \Lambda+\nu_k} \right| \le  C,
\]
hold uniformly over all $\Lambda \in (-\infty,0]$ and $k\in \mathbb{N}$. Finally, for smooth domains and polygons, \eqref{eq:conjB} complies, in the  leading term, with the asymptotic behaviour of $\sigma_k^{(\Lambda)}$ as $\Lambda\to-\infty$, see \cite[\S5.2]{GLP26} and references therein.

\begin{figure}[htb]
\centering
\includegraphics[width=0.8\textwidth]{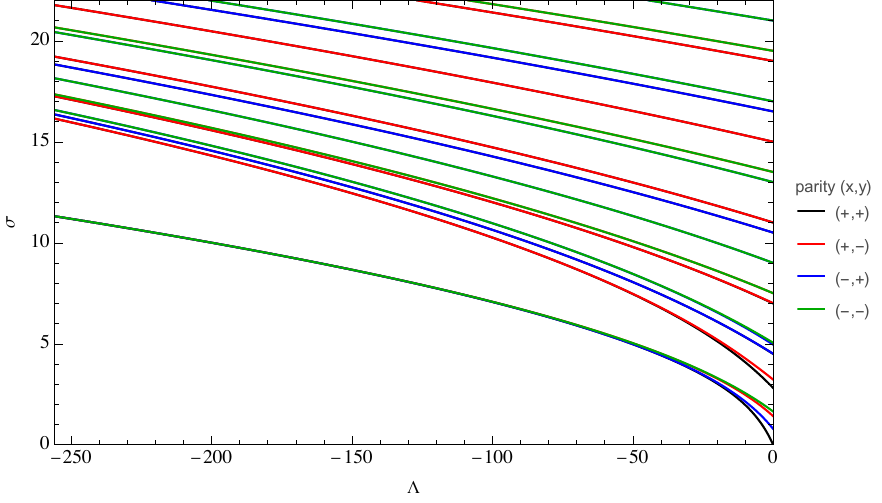}
\caption{The analytic eigenvalue curves $\sigma^{(\Lambda)}$, $\Lambda\le 0$, for the rectangle $\left(-\frac{\pi}{4},-\frac{\pi}{4}\right)\times \left(-\frac{\pi}{6},-\frac{\pi}{6}\right)$.  A curve colour reflects the parity of the corresponding bulk eigenfunction with respect to the central line in the given direction: $+$ is symmetric, $-$ is antisymmetric. For an algorithm and theoretical justification, see e.g. \cite{GLP26}.}\label{fig:rect}
\end{figure}

\subsection{Statements of the main results}

Our main result is the following quadratic-form comparison for \emph{convex} domains.  Recall that all forms $\mathfrak d_\Lambda$ have the common form domain $H^{1/2}(\partial\Omega)$ for $\Lambda\leq0$.

\begin{theorem}\label{thm:convex}
Let $\Omega\subset\R^d$ be a bounded convex domain and let $\Lambda_1\leq\Lambda_2\leq0$. Then
\begin{equation}\label{eq:formHolder}
0\leq \DtN_{\Lambda_1}-\DtN_{\Lambda_2}\leq \sqrt{\Lambda_2-\Lambda_1}\,I
\end{equation}
in the sense of quadratic forms. Equivalently, for every $u\in H^{1/2}(\partial\Omega)$,
\[
0\leq \mathfrak d_{\Lambda_1}[u]-\mathfrak d_{\Lambda_2}[u]\leq \sqrt{\Lambda_2-\Lambda_1}\,\|u\|^2_{L^2(\partial\Omega)}.
\]
\end{theorem}

Thus, for every bounded convex domain, the family $\Lambda\mapsto\DtN_\Lambda$ is $\frac12$-H\"older continuous on $(-\infty,0]$ with the H\"older constant one (which is sharp, see Remark \ref{rem:constantsharp}) in the sense that the form difference $\DtN_{\Lambda_1}-\DtN_{\Lambda_2}$ extends to a bounded operator on $L^2(\partial\Omega)$ satisfying the corresponding norm estimate.

The eigenvalue statement follows immediately from the min--max principle.

\begin{corollary}\label{cor:convex-eigenvalues}
Under the assumptions of Theorem~\ref{thm:convex}, for every $k\in\mathbb N$,
\begin{equation}\label{eq:Holder}
0\leq \sigma_k^{(\Lambda_1)}-\sigma_k^{(\Lambda_2)} \leq \sqrt{\Lambda_2-\Lambda_1}.
\end{equation}
In particular, taking $\Lambda_2=0$ gives \eqref{eq:conjB}.
\end{corollary}

The sharp H\"older constant one in the right-hand sides of \eqref{eq:formHolder} and \eqref{eq:Holder} is special for convex domains.  For a general smooth domain one still has a similar estimate but with a domain-dependent constant.  To state it, suppose that $\partial\Omega$ is of class $C^2$, and let $\mathcal H(x)$ denote the mean curvature of $\partial\Omega$ with respect to the outward unit normal, with the convention that $\mathcal H$ is the sum of the principal curvatures. Set
\begin{equation}\label{eq:KOmega}
\mathcal{K}_\Omega:=\|\mathcal H_-\|_{L^\infty(\partial\Omega)},\qquad \mathcal H_-:=\max\{-\mathcal H,0\},
\end{equation}
and
\begin{equation}\label{eq:AOmega}
\mathcal{A}_\Omega:=\sup_{0\neq u\in H^{1/2}(\partial\Omega)}
\frac{\|\mathcal E_0u\|^2_{L^2(\Omega)}}{\|u\|^2_{L^2(\partial\Omega)}}.
\end{equation}
The constant $\mathcal{A}_\Omega$ is finite, see also Remark \ref{rem:AOmega}. 

\begin{theorem}\label{thm:C2}
Let $\Omega\subset\R^d$ be a bounded domain with $C^2$ boundary. For any $\Lambda_1\leq\Lambda_2\leq0$,
\begin{equation}\label{eq:formHolderC2}
0\leq \DtN_{\Lambda_1}-\DtN_{\Lambda_2}\leq \mathcal{C}_\Omega\sqrt{\Lambda_2-\Lambda_1}\,I
\end{equation}
in the sense of quadratic forms, where one may take
\begin{equation}\label{eq:COmega}
\mathcal{C}_\Omega:=1+\mathcal{K}_\Omega \mathcal{A}_\Omega.
\end{equation}
\end{theorem}

\begin{corollary}\label{cor:C2-eigenvalues}
Under the assumptions of Theorem~\ref{thm:C2}, for every $k\in\mathbb N$,
\[
0\leq \sigma_k^{(\Lambda_1)}-\sigma_k^{(\Lambda_2)} \leq \mathcal{C}_\Omega\sqrt{\Lambda_2-\Lambda_1}.
\]
\end{corollary}

The same computation, with a different choice of the parameter $\alpha$, yields a second, complementary estimate, in which the geometric defect enters \emph{additively} rather than multiplicatively, and in which no auxiliary constant $\mathcal{A}_\Omega$ is needed. For $t\ge0$, set
\begin{equation}\label{eq:Theta}
\Theta_\Omega(t):=\frac{\mathcal{K}_\Omega+\sqrt{\mathcal{K}_\Omega^2+4t}}{2},
\end{equation}
that is, let $\Theta_\Omega(t)$ be the unique non-negative root $\alpha$ of the quadratic equation 
\[
\alpha^2 - \mathcal{K}_\Omega\alpha - t=0.
\]

\begin{theorem}\label{thm:main}
Let $\Omega\subset\R^d$ be a bounded domain with $C^2$ boundary. For any $\Lambda_1\leq\Lambda_2\leq0$,
\begin{equation}\label{eq:formHolderTheta}
0\leq \DtN_{\Lambda_1}-\DtN_{\Lambda_2}\leq \Theta_\Omega\left(\Lambda_2-\Lambda_1\right) I
\end{equation}
in the sense of quadratic forms. In particular, since $\Theta_\Omega(t)\le\sqrt t+\mathcal{K}_\Omega$,
\begin{equation}\label{eq:formHolderThetaweak}
0\leq \DtN_{\Lambda_1}-\DtN_{\Lambda_2}\leq \left(\sqrt{\Lambda_2-\Lambda_1}+\mathcal{K}_\Omega\right) I.
\end{equation}
\end{theorem}

\begin{corollary}\label{cor:main-eigenvalues}
Under the assumptions of Theorem~\ref{thm:main}, for every $k\in\mathbb N$,
\[
0\leq \sigma_k^{(\Lambda_1)}-\sigma_k^{(\Lambda_2)} \leq \Theta_\Omega\left(\Lambda_2-\Lambda_1\right) \leq \sqrt{\Lambda_2-\Lambda_1}+\mathcal{K}_\Omega.
\]
\end{corollary}

\begin{remark}\label{rem:compare}
Neither Theorem~\ref{thm:C2} nor Theorem~\ref{thm:main} implies the other:  Theorem~\ref{thm:C2} is the stronger statement near the diagonal $\Lambda_1=\Lambda_2$, and Theorem~\ref{thm:main} is the stronger one away from it. The two statements convey genuinely different information. Only \eqref{eq:formHolderC2} tends to zero as $\Lambda_1\to\Lambda_2$, and it is therefore the only one of the two bounds giving the $\frac12$-H\"older continuity. On the other hand,
\begin{equation}\label{eq:Kinf}
\Theta_\Omega(t)=\sqrt t+\frac{\mathcal{K}_\Omega}{2}+O\left(t^{-1/2}\right)\qquad\text{as }t\to+\infty,
\end{equation}
so that \eqref{eq:formHolderTheta} retains the sharp leading constant one of Theorem~\ref{thm:convex}, at the cost of an additive defect. This is the regime relevant to \eqref{eq:conjB} for large $-\Lambda$, and there the additive constant $\frac{\mathcal{K}_\Omega}{2}$ cannot be improved, see Remark \ref{rem:infty}.
\end{remark}

\begin{remark}\label{rem:AOmega}
The constant $\mathcal{A}_\Omega$ of \eqref{eq:AOmega} admits an explicit bound in terms of the classical  \emph{torsion function} $\psi$ of $\Omega$, i.e.\ the solution of 
\[
-\Delta\psi=1\quad\text{in }\Omega, \qquad \psi|_{\partial\Omega}=0. 
\]
Indeed, for $U_0:=\mathcal E_0u$ one has $\Delta|U_0|^2=2|\nabla U_0|^2$, and evaluating $\int_\Omega\nabla\psi\cdot\nabla|U_0|^2\,\dr x$ by Green's formula in the two possible ways (the boundary term $\int_{\partial\Omega}\psi\,\partial_n|U_0|^2\,\dr S$ vanishes because $\psi|_{\partial\Omega}=0$) gives the identity
\begin{equation}\label{eq:torsionidentity}
\|U_0\|^2_{L^2(\Omega)}+2\int_\Omega\psi\,|\nabla U_0|^2\,\dr x =\int_{\partial\Omega}|u|^2\left(-\partial_n\psi\right)\dr S.
\end{equation}
Since $\psi\ge0$, and since taking $u\equiv1$ in \eqref{eq:torsionidentity} recovers $\int_{\partial\Omega}\left(-\partial_n\psi\right)\dr S=\int_\Omega(-\Delta\psi)\,\dr x=|\Omega|$, we obtain
\begin{equation}\label{eq:AOmegabounds}
\frac{|\Omega|}{|\partial\Omega|} \le  \mathcal{A}_\Omega \le  \max_{\partial\Omega}\left|\nabla\psi\right|.
\end{equation}
Both inequalities become equalities for a ball $B(R)\subset\R^d$, for which $\psi=\frac{R^2-|x|^2}{2d}$ and $\mathcal{A}_{B(R)}=\frac Rd=\frac{|B(R)|}{|\partial B(R)|}$, the supremum in \eqref{eq:AOmega} being attained at constants. We note that $\psi$ is also known as the \emph{landscape function} \cite{FilocheMayboroda}, and that $\mathcal{A}_\Omega^{-1/2}$ is the minimal \emph{tension} at Helmholtz parameter $0$ in the sense of the method of particular solutions, cf.\ \cite{Barnett, BarnettHassell}; the estimates available there are, however, high-frequency ones, and carry no explicit constants.
\end{remark}

We say that a $C^2$ domain is \emph{weakly mean-convex} if $\mathcal H\geq0$ on $\partial\Omega$.  In this case $\mathcal{K}_\Omega=0$, and Theorems~\ref{thm:C2} and \ref{thm:main} coincide and give back the sharp constant one.  In dimension two weak mean convexity is equivalent to convexity, while in dimensions $d\geq3$ it is a strictly weaker condition. The following result shows that the convexity condition of Theorem~\ref{thm:convex} can be relaxed for sufficiently smooth domains for $d\ge 3$.

\begin{corollary}\label{cor:meanconvex}
Let $\Omega\subset\R^d$ be a bounded weakly mean-convex domain with $C^2$ boundary. Then, for $\Lambda_1\leq\Lambda_2\leq0$,
\[
0\leq \DtN_{\Lambda_1}-\DtN_{\Lambda_2} \leq \sqrt{\Lambda_2-\Lambda_1}\,I
\]
in the sense of quadratic forms, and consequently \eqref{eq:Holder} holds for every $k\in\mathbb N$.
\end{corollary}

\begin{remark}\label{rem:constantsharp}
The constant one in Theorem~\ref{thm:convex} and Corollary~\ref{cor:meanconvex} is sharp.  Indeed, already for a ball, with $\Lambda=-q^2$, the principal Dirichlet-to-Neumann eigenvalue satisfies $\sigma_1^{(-q^2)}=q+O(1)$ as $q\to+\infty$, see, e.g., \cite[\S5.2]{GLP26}. Since $\sigma_1^{(0)}=0$, no constant smaller than one can hold uniformly in the comparison with $\Lambda_2=0$.
\end{remark}

The $C^2$ hypothesis in Theorems~\ref{thm:C2} and \ref{thm:main} is used only through the explicit constant \eqref{eq:KOmega}: both results extend to bounded $C^{1,1}$ domains, with $\mathcal K_\Omega$ replaced by a constant controlling the distributional Laplacian of the distance to the boundary, see Remark~\ref{rem:C11}.

\begin{remark}
There is also a  Riemannian analogue of Theorem~\ref{thm:C2}, which is stated in Theorem~\ref{thm:riemannian-general} below.  It shows that the domain-dependent $\frac12$-H\"older estimate is not specifically Euclidean; the sharp constant one is recovered under non-negative Ricci curvature and weak mean convexity of the boundary.
\end{remark}

We emphasise that the sharp constant-one comparison does not extend to arbitrary non-convex domains.  Moreover, the following examples show that both Conjectures~\ref{conj:A} and \ref{conj:B} may fail without an appropriate geometric assumption.

Let
\[
\begin{alignedat}{3}
B(x_0,\ell)&:=\{x\in\mathbb{R}^2: |x-x_0|<\ell\},&\quad B(\ell)&:=B(0, \ell),&\quad\mathbb{D}&:=B(1),\\
C(x_0, \ell)&:=\{x\in\mathbb{R}^2: |x-x_0|=\ell\},&\quad C(\ell)&:=C(0, \ell),&\quad\mathbb{S}&:=C(1),\\
A(x_0, \ell_1, \ell_2)&:=B(x_0,\ell_2)\setminus \overline{B(x_0,\ell_1)},&\quad A(\ell_1, \ell_2)&:=A(0, \ell_1, \ell_2),&\quad A(\ell)&:=A(1,\ell),
\end{alignedat}
\]
denote disks, circles, and annuli in the plane.

\begin{theorem}\label{thm:acounter}
Let $\ell_*\approx 9.1863$ be the unique root of the equation $\ell_*\log(\ell_*)=2(\ell_*+1)$. For any  annulus $A(\ell)$ with $\ell>\ell_*$,  there exists $\Lambda_\ell<0$ and  an eigenvalue branch which contradicts Conjecture \ref{conj:A} for all negative $\Lambda<\Lambda_\ell$. 
\end{theorem}

A numerical example indicates that one can take in Theorem \ref{thm:acounter} $\ell=10$ and $\Lambda_{\ell}=-23$, see Figure~\ref{fig:annulus}.

We also have a counterexample to Conjecture \ref{conj:B}. 
In order to formulate it, we describe a family of what we call \emph{Swiss cheese domains}.
Take $0<\epsilon<\frac14$, and choose a maximal family of points
$x_1,\dots,x_{m_\epsilon}\in B\left(\frac34-\epsilon\right)$ such that the disks
$B(x_j,\epsilon)\subset B\left(\frac34\right)$ are pairwise disjoint. By maximality, the disks
$B(x_j,2\epsilon)$ cover $B\left(\frac34-\epsilon\right)$, and hence
\[
4m_\epsilon\epsilon^2 \ge \left(\frac34-\epsilon\right)^2 \ge \frac14,
\]
hence,
\begin{equation}\label{eq:bcounter-packing}
m_\epsilon\ge\frac{1}{16\epsilon^2}.
\end{equation}

Let now 
\begin{equation}\label{eq:delta}
\delta:=\er^{-4},\qquad r_\epsilon:=\epsilon\delta=\epsilon\er^{-4}, 
\end{equation}
and define
\[
\Omega_\epsilon := \mathbb{D}\setminus \bigcup_{j=1}^{m_\epsilon}\overline{B(x_j, r_\epsilon)}.
\]
The domain $\Omega_\epsilon$ is smooth, bounded, and connected, see Figure \ref{fig:cheese}.

\begin{figure}[htb]
\centering
\includegraphics[width=\linewidth]{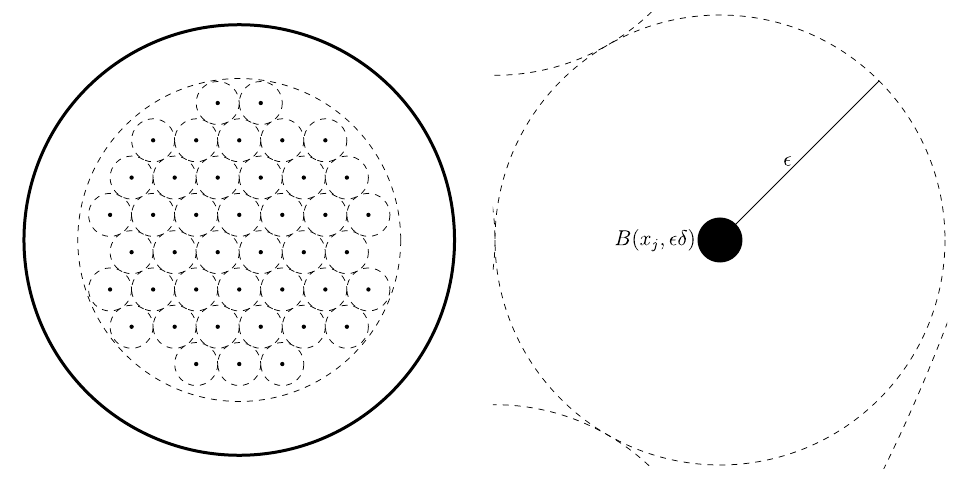}
\caption{The domain $\Omega_\epsilon$ (left), and its zoom near a single hole (right), not to scale. The dashed lines are auxiliary.}\label{fig:cheese}
\end{figure}

\begin{theorem}[Swiss cheese domains]\label{thm:bcounter}
For all sufficiently small $\epsilon>0$, there exists $q_\epsilon>0$ such that
\[
\sigma_{m_\epsilon}^{(-q_\epsilon^2)}(\Omega_\epsilon) - \sigma_{m_\epsilon}^{(0)}(\Omega_\epsilon) > q_\epsilon.
\]
\end{theorem}

\begin{remark}
In three dimensions, similar counterexamples $\Omega_\epsilon \subset \R^3$ exist, obtained by deleting numerous small  solid tori (rather than small balls) from the unit ball. By additionally cutting out thin channels joining each toroidal boundary with the outer boundary component of $\partial \Omega$, we can even produce a counterexample $\widetilde{\Omega}_\epsilon$ to Conjecture~\ref{conj:B} with connected boundary $\partial \widetilde{\Omega}_\epsilon$. We omit the details. However, in dimension two it is still an open question whether Conjectures \ref{conj:A} and \ref{conj:B} are true for \emph{simply connected non-convex planar} domains.
\end{remark} 

\subsection{Organisation of the paper}
In \S\ref{sec:prel}, we recall some additional auxiliary results needed for subsequent proofs. The proofs of main bounds of Theorems \ref{thm:convex}, \ref{thm:C2}, and \ref{thm:main} are presented in \S\ref{sec:proofconvex}. 
The crucial point is the use of test functions expressed in terms of exponentials of the distance function to the boundary, and the specific properties of the distance function for convex domains.
The counterexample Theorems \ref{thm:acounter} and \ref{thm:bcounter} are proved in \S\ref{sec:counter}. Various extensions are discussed in \S\ref{sec:extensions}: we briefly mention the consequences for the bounds on Robin eigenvalues, discuss the cases of Dirichlet-to-Neumann maps on Riemannian manifolds and on metric graphs, and provide an example showing that Theorems \ref{thm:C2} and \ref{thm:main} cannot be extended beyond $C^{1,1}$ regularity of the boundary, cf.\ Remark \ref{rem:C11}.

\section{Preliminaries}\label{sec:prel}
\subsection{The Robin Laplacian and Robin--Dirichlet-to-Neumann duality}\label{sec:RobDtN}

Let $\gamma\in\mathbb{R}$. We denote by $-\Delta^{\Rob,\gamma}_\Omega$ the Robin Laplacian in a bounded open set $\Omega$ with Lipschitz boundary (subject to the boundary condition $\partial_n U + \gamma U=0$ on $\partial\Omega$) defined as the self-adjoint operator  associated with the closed semi-bounded below form
\begin{equation}\label{eq:robinform}
\mathfrak{r}_\gamma[U]:=\|\nabla U\|^2_{L^2(\Omega)} + \gamma  \left\|U|_{\partial\Omega}\right\|^2_{L^2(\partial \Omega)}.
\end{equation}
Its eigenvalues, listed with multiplicities, will be denoted by 
\[
\lambda^{\Rob,\gamma}_1(\Omega)\le \lambda^{\Rob,\gamma}_2(\Omega)\le \dots.
\]
The eigenvalues of the Robin Laplacian are strictly monotone increasing functions of $\gamma$. Moreover, the union of the spectra $\bigcup_{\gamma\in\mathbb{R}}\Spec\left(-\Delta^{\Rob,\gamma}\right)$ can be decomposed into a family of real-analytic curves $\lambda^{\Rob,\gamma}$, once again at the cost of ordering. For further details, see \cite{Bucur17}. 

The spectra of the Dirichlet-to-Neumann map $\DtN_\Lambda$ and the Robin Laplacian $-\Delta^{\Rob,\gamma}$ are closely related via the so-called \emph{Robin--Dirichlet-to-Neumann duality}, see further \cite{Arendt12}, \cite{Friedlander1991}, \cite{LMP}, \cite{Hassannezhad22}.

\begin{prop}\label{prop:DtNRduality}
Let $\Omega\subset\mathbb{R}^d$ be a bounded Lipschitz domain, and let $\Lambda, \sigma\in\mathbb{R}$. Then $\sigma\in\Spec\left(\DtN_\Lambda\right)$ if and only if $\Lambda\in\Spec\left(-\Delta^{\Rob,-\sigma}\right)$. Moreover, the multiplicities of $\sigma$ as an eigenvalue of $\DtN_\Lambda$ and $\Lambda$ as an eigenvalue of $-\Delta^{\Rob,-\sigma}$ coincide, and $\mathcal{E}_\Lambda u$ is an eigenfunction of the  Robin Laplacian $-\Delta^{\Rob,-\sigma}$ if and only if $u$ is an eigenfunction of $\DtN_\Lambda$.

Furthermore, if, additionally, $\Lambda<\lambda_1^\Dir(\Omega)$ and $\sigma = \sigma_k^{(\Lambda)}(\Omega)$, $k\in\mathbb{N}$, then 
\[
\Lambda=\lambda^{\Rob, -\sigma}_k(\Omega).
\]
\end{prop}

\begin{definition}\label{defn:counting}
For a self-adjoint semibounded below operator $\mathfrak A$ with a discrete spectrum of eigenvalues $\lambda_1\le \lambda_2\le \dots$, we define its \emph{eigenvalue counting function} as 
$\mathcal{N}^{\mathfrak A}: \mathbb{R}\to \{0\}\cup\mathbb{N}$, 
\[
\mathcal{N}^{\mathfrak A}(t):=\#\left\{k: \lambda_k\le t\right\}.
\]
\end{definition}

As an immediate corollary of Proposition \ref{prop:DtNRduality}, we deduce that  for any $\Lambda\le 0$ and any $\sigma\in\mathbb{R}$, 
\begin{equation}\label{eq:DtNRdualityN}
\mathcal{N}^{\DtN_\Lambda}(\sigma) = \mathcal{N}^{-\Delta^{\Rob,-\sigma}}(\Lambda).
\end{equation}
Note that for $\sigma<0$, \eqref{eq:DtNRdualityN} becomes a trivial identity $0=0$. 

\subsection{Glazman's Lemma}

We will make a frequent use of the following classical result. 

\begin{lemma}[\cite[\S~III]{ReedSimonIV}, \cite[Proposition 9.5]{Shubin20}, \cite[Lemma 3.2.31]{LMP}]\label{lem:glazman}
Let $\mathfrak a$ be a densely defined, closed, lower-semibounded symmetric form with
form domain $\Dom(\mathfrak{a})$ in a Hilbert space $\mathfrak H$, whose associated
self-adjoint operator $\mathfrak A$ has purely discrete spectrum
$\lambda_1\le\lambda_2\le\cdots$. For $\lambda\in\R$,
\[
\mathcal{N}^{\mathfrak A}(\lambda)=\#\{k\in\mathbb{N}:\lambda_k\le \lambda\}
=\max_{\substack{\mathcal{L}\subset\Dom(\mathfrak{a}):\\\mathfrak{a}[u]\le \lambda\|u\|_{\mathfrak H}^2\text{ for all }u\in \mathcal{L}}} \dim\mathcal{L},
\]
where $\mathcal{L}$ is a finite-dimensional linear subspace of $\Dom(\mathfrak{a})$. The same result holds if both  non-strict inequalities are replaced by strict ones.
\end{lemma}

\section{Proofs of main results}\label{sec:proofconvex}
\subsection{The main idea}\label{subsec:mainidea}

Let, for a bounded domain $\Omega\subset\mathbb{R}^d$,
\[
\rho(x):=\dist(x,\partial\Omega),
\]
be the distance function to the boundary. We recall some of its well-known properties;
see, e.g., \cite[Chapter~2]{BalinskyEvansLewis} for details.
The function $\rho$ is $1$-Lipschitz and satisfies
$|\nabla\rho|=1$ almost everywhere in $\Omega$. Throughout this section, we write
\[
\mu_\rho:=-\Delta\rho\in\mathscr D'(\Omega),
\]
where the Laplacian is understood in the sense of distributions.

The only analytic input in all that follows is a distributional lower bound on $\mu_\rho$. We therefore set
\begin{equation}\label{eq:MOmega}
\mathcal M_\Omega:=\inf\left\{m\ge 0:\ \mu_\rho\ge -m\,\dr x\ \text{ in }\mathscr D'(\Omega)\right\}\in[0,+\infty],
\end{equation}
with the convention $\inf\emptyset=+\infty$. Whenever $\mathcal M_\Omega<\infty$ the infimum in \eqref{eq:MOmega} is attained, so that
\begin{equation}\label{eq:MOmegabound}
\mu_\rho\ge -\mathcal M_\Omega\,\dr x\qquad\text{in }\mathscr D'(\Omega).
\end{equation}
In contrast with the curvature quantity $\mathcal K_\Omega$ of \eqref{eq:KOmega}, which requires $\partial\Omega\in C^2$, the constant $\mathcal M_\Omega$ is defined for an arbitrary bounded domain. The two are compared in \eqref{eq:distlowerC2} below.

If $\Omega$ is convex, then $\rho$ is concave in $\Omega$; hence
\begin{equation}\label{eq:convexdistmain}
\mu_\rho=-\Delta\rho\geq 0\qquad\text{in }\mathscr D'(\Omega).
\end{equation}
If instead $\partial\Omega$ is of class $C^2$, \cite[Theorem~1.6]{LewisLiLi} gives, in the notation of \eqref{eq:KOmega},
\[
-\Delta\rho\geq
\frac{(d-1)\mathcal H(\pi(x))}{(d-1)-\rho(x)\mathcal H(\pi(x))} \geq-\mathcal{K}_\Omega \qquad\text{in }\mathscr D'(\Omega),
\]
where $\pi(x)$ is a nearest boundary point on the regular set of the distance function.  In particular, comparing with \eqref{eq:MOmega},
\begin{equation}\label{eq:distlowerC2}
\mathcal M_\Omega\le \mathcal K_\Omega,\qquad\text{that is,}\qquad
\mu_\rho\geq-\mathcal{K}_\Omega\,\dr x \quad\text{in }\mathscr D'(\Omega).
\end{equation}
By \eqref{eq:convexdistmain}, $\mathcal M_\Omega=0$ for a convex domain; when $\Omega$ is weakly mean-convex, $\mathcal{K}_\Omega=0$ and \eqref{eq:distlowerC2} gives $\mathcal M_\Omega=0$ as well.

In the convex case, the interior trace of $\nabla\rho$ on $\partial\Omega$ is
$-n$ almost everywhere, where $n$ is the outward unit normal. In the $C^2$ case, $\rho$ is $C^2$ in a tubular neighbourhood of the boundary and $\nabla\rho=-n$ classically on $\partial\Omega$. These boundary properties are justified in \S\ref{sec:dist}.

The central estimate is the following slightly more general form of the substitution inequality.

\begin{prop}\label{prop:sub}
Let $\Omega\subset\R^d$ be a bounded domain with $\mathcal M_\Omega<\infty$ for which the Green formula \eqref{eq:Greenrho} holds; this is the case, in particular, if $\Omega$ is convex, or if $\partial\Omega$ is of class $C^{1,1}$.  For $\gamma\in\R$, $\alpha\geq0$, and $U\in H^1(\Omega)$, set
\begin{equation}\label{eq:VU}
V(x):=\er^{-\alpha\rho(x)}U(x),\qquad x\in\Omega.
\end{equation}
Then $V\in H^1(\Omega)$, $V|_{\partial\Omega}=U|_{\partial\Omega}$, and
\begin{equation}\label{eq:sub}
\mathfrak r_{\gamma-\alpha}[V]+\alpha^2\|V\|^2_{L^2(\Omega)} \leq \mathfrak r_\gamma[U]+\alpha \mathcal M_\Omega\|V\|^2_{L^2(\Omega)}.
\end{equation}
For convex or $C^2$ weakly mean-convex domains one has $\mathcal M_\Omega=0$, recovering the sharp form of the estimate.
\end{prop}
Proposition~\ref{prop:sub} is proved in \S\ref{sec:dist}.

\begin{remark}\label{rem:C11}
The $C^2$ boundary  assumption in Theorems~\ref{thm:C2} and \ref{thm:main} can be further relaxed to allow $C^{1,1}$ boundaries as long as $\mathcal K_\Omega$ is replaced by $\mathcal M_\Omega$ in \eqref{eq:COmega} and \eqref{eq:Theta}, we omit the details. However, lowering the regularity beyond $C^{1.1}$ may make the main statements useless, see \S\ref{sec:wiggly}.
\end{remark}

\subsection{Proofs of Theorems \ref{thm:convex}, \ref{thm:C2}, and \ref{thm:main}}

Assume $\Lambda_1\leq\Lambda_2\leq0$ and write
\[
q_j:=\sqrt{-\Lambda_j},\qquad
s:=\Lambda_2-\Lambda_1=q_1^2-q_2^2 \ge  0.
\]
We first note that the lower form inequality does not require any geometric assumption.  Indeed, by the Dirichlet principle \eqref{eq:DirPr}, for every $u\in H^{1/2}(\partial\Omega)$,
\begin{equation}\label{eq:formmonotonicity}
\mathfrak d_{\Lambda_2}[u]\leq \mathfrak d_{\Lambda_1}[u],
\end{equation}
since $q_2^2\leq q_1^2$.

All three upper bounds come from one and the same computation, in which the exponent $\alpha$ of the substitution \eqref{eq:VU} is left \emph{free} and is fixed only at the very last step; the three theorems correspond to three different choices of $\alpha$.

For $u\in H^{1/2}(\partial\Omega)$ and $q\ge 0$,  we will use the shorthand notation
\[
U_q = \mathcal E_{-q^2}u.
\]
Let $\alpha\ge0$ be arbitrary, and put
\[
V:=\er^{-\alpha\rho}U_{q_2}.
\]
Since $V|_{\partial\Omega}=u$, $W:=V$ is an admissible argument in the Dirichlet principle \eqref{eq:DirPr} for $\mathfrak d_{\Lambda_1}[u]$. Applying Proposition~\ref{prop:sub} with $\gamma=\alpha$ gives
\[
\int_\Omega |\nabla V|^2\,\dr x+\alpha^2\|V\|^2_{L^2(\Omega)}
\leq \int_\Omega |\nabla U_{q_2}|^2\,\dr x + \alpha\|u\|^2_{L^2(\partial\Omega)} + \alpha \mathcal{M}_\Omega\|V\|^2_{L^2(\Omega)}.
\]
Adding $q_1^2\|V\|^2_{L^2(\Omega)}$ to both sides, writing $q_1^2=q_2^2+s$, and using $|V|\leq|U_{q_2}|$ to estimate $q_2^2\|V\|^2_{L^2(\Omega)}\le q_2^2\|U_{q_2}\|^2_{L^2(\Omega)}$, we obtain, for every $\alpha\ge0$,
\begin{equation}\label{eq:formgeneral-intermediate}
\begin{split}
\mathfrak d_{\Lambda_1}[u] &\leq \int_\Omega\left(|\nabla V|^2+q_1^2|V|^2\right)\,\dr x\\
&\leq \mathfrak d_{\Lambda_2}[u] +\alpha\|u\|^2_{L^2(\partial\Omega)} +\left(s+\alpha \mathcal{M}_\Omega-\alpha^2\right)\|V\|^2_{L^2(\Omega)}.
\end{split}
\end{equation}

\begin{proof}[Proof of Theorem~\ref{thm:convex}] Here $\Omega$ is convex, hence $\mathcal{M}_\Omega=0$ by \eqref{eq:convexdistmain}, and the choice $\alpha=\sqrt s$ annihilates the last term in \eqref{eq:formgeneral-intermediate}. Together with \eqref{eq:formmonotonicity} this completes the proof.
\end{proof}

\begin{proof}[Proof of Theorem~\ref{thm:main}] 
Let now $\partial\Omega$ be of class $C^2$, write $\Theta_\Omega^{\mathcal M}$ for the function \eqref{eq:Theta} with $\mathcal K_\Omega$ replaced by $\mathcal M_\Omega$, and choose
\[
\alpha:=\Theta^{\mathcal M}_\Omega(s)\ \ge 0.
\]
By the definition of $\Theta^{\mathcal M}_\Omega(s)$ as the non-negative root of $\alpha^2-\mathcal{M}_\Omega\alpha - s=0$, the coefficient $s+\alpha \mathcal{M}_\Omega-\alpha^2$ in \eqref{eq:formgeneral-intermediate} vanishes identically, and \eqref{eq:formgeneral-intermediate} reduces to
\[
\mathfrak d_{\Lambda_1}[u]-\mathfrak d_{\Lambda_2}[u]
\le \Theta^{\mathcal M}_\Omega(s)\,\|u\|^2_{L^2(\partial\Omega)}.
\]
Since the right-hand side of \eqref{eq:Theta} is non-decreasing in the constant appearing in it, \eqref{eq:distlowerC2} gives $\Theta^{\mathcal M}_\Omega(s)\le\Theta_\Omega(s)$; this, together with \eqref{eq:formmonotonicity}, proves \eqref{eq:formHolderTheta}, the weaker form \eqref{eq:formHolderThetaweak} follows from $\sqrt{\mathcal{K}_\Omega^2+4s}\le \mathcal{K}_\Omega+2\sqrt s$. Note that here the error term is absorbed into the spectral parameter, at the cost of making $\alpha$ larger, rather than estimated separately; in particular no analogue of $\mathcal{A}_\Omega$ is required, and the argument uses no boundary regularity beyond that already needed for the finiteness of $\mathcal M_\Omega$, cf. Remark~\ref{rem:C11}.
\end{proof}

\begin{proof}[Proof of Theorem~\ref{thm:C2}] Choose instead $\alpha=\sqrt s$, so that the coefficient in \eqref{eq:formgeneral-intermediate} equals $\alpha \mathcal{M}_\Omega\ge0$ and, using once more $|V|\le|U_{q_2}|$ and then $\mathcal M_\Omega\le\mathcal K_\Omega$,
\begin{equation}\label{eq:formgeneral-intermediate-C2}
\mathfrak d_{\Lambda_1}[u] \leq \mathfrak d_{\Lambda_2}[u] +\alpha\|u\|^2_{L^2(\partial\Omega)} +\alpha \mathcal{K}_\Omega\|U_{q_2}\|^2_{L^2(\Omega)}.
\end{equation}
It remains to estimate $\|U_{q_2}\|_{L^2(\Omega)}$, and this is where the constant $\mathcal{A}_\Omega$ from \eqref{eq:AOmega} enters.  

We first prove that $\mathcal{A}_\Omega$ is finite, or, more precisely, that 
the upper bound in \eqref{eq:AOmegabounds} is finite under the regularity
assumptions used here. Indeed, since $\partial\Omega$ is of class $C^2$,
the global elliptic
regularity for the Dirichlet problem gives $\psi\in W^{2,p}(\Omega)$ or every $1<p<\infty$,
see, for instance, \cite[Theorem~9.15]{GilbargTrudinger}. Choosing
$p>d$ and using Sobolev embedding, we obtain
$\psi\in C^{1}(\overline\Omega)$. Hence
$\max_{\partial\Omega}|\nabla\psi|<\infty$, and
\eqref{eq:AOmegabounds} in particular shows that $\mathcal{A}_\Omega<\infty$.

The same constant controls the $L^2$ norm of every negative-parameter harmonic extension $U_q=\mathcal E_{-q^2}u$.  Indeed, 
\[
U_q - U_0=-q^2\left(-\Delta_\Omega^\Dir+q^2\right)^{-1} U_0,
\]
so that
\begin{equation}\label{eq:L2extension}
U_q=-\Delta_\Omega^\Dir\left(-\Delta_\Omega^\Dir+q^2\right)^{-1}U_0,
\qquad \|U_q\|_{L^2(\Omega)}\leq\|U_0\|_{L^2(\Omega)} \leq \sqrt{\mathcal{A}_\Omega}\,\|u\|_{L^2(\partial\Omega)}.
\end{equation}
Substituting \eqref{eq:L2extension} into \eqref{eq:formgeneral-intermediate-C2} gives
\[
\mathfrak d_{\Lambda_1}[u]-\mathfrak d_{\Lambda_2}[u]
\leq \alpha(1+\mathcal{K}_\Omega \mathcal{A}_\Omega)\|u\|^2_{L^2(\partial\Omega)}
= \mathcal{C}_\Omega\sqrt{\Lambda_2-\Lambda_1}\,\|u\|^2_{L^2(\partial\Omega)},
\]
which, together with \eqref{eq:formmonotonicity}, proves Theorem~\ref{thm:C2}.  
\end{proof}

Corollaries~\ref{cor:convex-eigenvalues}, \ref{cor:C2-eigenvalues}, \ref{cor:main-eigenvalues}, and \ref{cor:meanconvex} follow immediately from the min--max principle and from $\mathcal{K}_\Omega=0$ in the weakly mean-convex case.

\subsection{Properties of the $\dist$ function and the proof of Proposition \ref{prop:sub}}\label{sec:dist}

\begin{proof}[Proof of Proposition~\ref{prop:sub}]
The proof is essentially just integration by parts. To see this, let us first argue formally, assuming that $\rho$ and $U$ are regular enough. With $\mu_\rho=-\Delta\rho$ as introduced in \S\ref{subsec:mainidea}, set
\[
g=|U|^2,\qquad w=\er^{-2\alpha\rho}.
\]
By \eqref{eq:MOmegabound}, under the hypotheses of Proposition~\ref{prop:sub} we have
\[
\mu_\rho\geq- \mathcal{M}_\Omega\,\dr x.
\]
Moreover, $|\nabla\rho|=1$ almost everywhere in $\Omega$, $\partial_n\rho=-1$ on $\partial\Omega$ in the appropriate trace sense, and $w=1$ on $\partial\Omega$. Integration by parts therefore yields
\begin{equation}\label{eq:GGrho-cross}
\int_\Omega w\nabla g\cdot\nabla\rho\,\dr x = \int_\Omega wg\,\dr\mu_\rho
+2\alpha\int_\Omega wg\,\dr x - \int_{\partial\Omega}g\,\dr S.
\end{equation}
Since $\nabla g=2\operatorname{Re}(\overline U\nabla U)$, we have
\[
|\nabla V|^2
=w\left(|\nabla U|^2-\alpha\nabla g\cdot\nabla\rho+\alpha^2g\right),
\]
and hence
\[
\begin{split}
\mathfrak r_{\gamma-\alpha}[V]+\alpha^2\|V\|_{L^2(\Omega)}^2 &=\int_\Omega w|\nabla U|^2\,\dr x
+\gamma\int_{\partial\Omega}g\,\dr S - \alpha\int_\Omega wg\,\dr\mu_\rho\\
&\leq \mathfrak r_\gamma[U]+\alpha\mathcal{M}_\Omega\int_\Omega wg\,\dr x=\mathfrak r_\gamma[U]+\alpha \mathcal{M}_\Omega\|V\|^2_{L^2(\Omega)}.
\end{split}
\]
This proves the formal version of \eqref{eq:sub}.  The only geometric input in the inequality is the distributional lower bound on $\mu_\rho$.

It remains to justify the Green formula \eqref{eq:GGrho-cross}. Assume first that $\Omega$ is convex. Let $h_{\overline\Omega}$ be the support function of $\overline\Omega$, and set
\[
b(x):=\inf_{\theta\in\mathbb S^{d-1}} \left(h_{\overline\Omega}(\theta)-x\cdot\theta\right).
\]
Then $b$ is a concave and $1$-Lipschitz extension of $\rho$ to all of $\mathbb R^d$. Since $b=\rho$ in $\Omega$, the restriction $(-\Delta b)|_\Omega$ is precisely the distribution $\mu_\rho$. Applying \cite[Theorem~6.8]{EvansGariepy} to the convex function $-b$, we obtain $\nabla b\in BV(\Omega)$ and recover that $\mu_\rho$ is a finite nonnegative Radon measure. Fix a point $z\in\partial\Omega$ such that the unit normal $n(z)$ is unique, and consider a sequence $\Omega\ni x_j\to z$ of differentiability points of $b$. At each such point $\nabla b(x_j)=-\theta_j$, where $\theta_j$ is a minimiser in the definition of $b(x_j)$. Passing to a subsequence and using the uniqueness of the normal gives $\theta_j\to n(z)$, and therefore
\[
\lim_{x\to z}\nabla b(x)=-n(z)
\]
through differentiability points. Since the outward normal is unique almost everywhere on $\partial\Omega$, the interior trace of $\nabla b$ is $-n$ by the $BV$ trace characterisation \cite[Theorem~5.7]{EvansGariepy}. The integration by parts formula \cite[Theorem~5.6]{EvansGariepy}, extended from $C^1$ to Lipschitz test functions by approximation, now gives
\begin{equation}\label{eq:Greenrho}
\int_\Omega\varphi\,\dr\mu_\rho =\int_\Omega\nabla\varphi\cdot\nabla\rho\,\dr x + \int_{\partial\Omega}\varphi\,\dr S,
\qquad \varphi\in\operatorname{Lip}(\overline\Omega).
\end{equation}
For smooth $U$, taking $\varphi=wg$ gives \eqref{eq:GGrho-cross}; the extension to arbitrary $U\in H^1(\Omega)$ follows by approximation, since $\er^{-\alpha\rho}$ is an $H^1$ multiplier.

Assume next that $\partial\Omega$ is of class $C^2$.  By \eqref{eq:distlowerC2}, $\mu_\rho+\mathcal{K}_\Omega\,\dr x$ is a nonnegative Radon measure.  Moreover, $\rho$ is $C^2$ in a tubular neighbourhood of $\partial\Omega$, so $\mu_\rho$ has a bounded density there; on the compact complement of a smaller collar it has finite mass.  Thus $\mu_\rho$ is a finite signed Radon measure.  We also have $\nabla\rho=-n$ on the boundary.  Formula \eqref{eq:Greenrho} follows by applying the distributional identity for $\mu_\rho$ to cutoffs supported away from $\partial\Omega$ and then letting the cutoff approach the boundary (equivalently, by integrating over the inner parallel domains $\{\rho>t\}$ and letting $t\to0^+$).  The formal calculation above then proves \eqref{eq:sub}, and the extension from smooth $U$ to $H^1(\Omega)$ is again by approximation.
\end{proof}

\section{Counterexamples}\label{sec:counter}

The counterexamples of this section are motivated by the following calculation for a related, albeit somewhat different, problem. Namely, we consider the Dirichlet-to-Neumann map for the Helmholtz equation   
\begin{equation}\label{eq:Helmholtz}
-\Delta U = \Lambda U
\end{equation}
with $\Lambda<0$ in the \emph{exterior} $\Omega^\mathrm{ext} = \mathbb{R}^d\setminus\overline{\Omega}$ of a bounded Lipschitz domain $\Omega\subset\mathbb{R}^d$. For $\Lambda<0$, the operator can be defined in the same variational  manner as for bounded domains; in the case $\Lambda=0$, see \cite{external} for details.

If we take $\Omega=\mathbb{D}\in\mathbb{R}^2$, then the eigenvalues of $\DtN_\Lambda\left(\mathbb{D}^\mathrm{ext}\right)$ are easily found by separation of variables. In particular, the bottom eigenvalue always corresponds to a radial eigenfunction and is given, in terms of  modified Bessel functions $K_\nu$, by
\[
\sigma_1^{(\Lambda)}\left(\mathbb{D}^\mathrm{ext}\right)=
\begin{cases}
\frac{q K_1(q)}{K_0(q)}&\quad\text{if }\Lambda=-q^2, q>0,\\
0,&\quad\text{if }\Lambda=0.
\end{cases}
\]
It is easily seen \cite[eq. (17a)]{GrCh} that $\sigma_1^{(-q^2)}-\sigma_1^{(0)}=-\frac{1}{\log q}(1+o(1))>q$ as $q\to 0^+$, in contrast to validity of \eqref{eq:conjB} with $k=1$ for bounded domains.  

\begin{remark} 
We note that, according to \cite{GrCh}, both Conjectures \ref{conj:A} and \ref{conj:B} remain valid  for the exterior of the unit ball in dimensions three and higher.
\end{remark}

\subsection{Proof of  Theorem \ref{thm:acounter}}

A general radial solution of the Helmholtz equation \eqref{eq:Helmholtz} with $\Lambda<0$ in the annulus $A(\ell)$ is given, in terms of modified Bessel functions $I_0$ and $K_0$,  by 
\[
U(r) = c_1 I_0(q r)  + c_2 K_0(q r), \qquad q=\sqrt{-\Lambda},
\]
and a routine computation shows that there are two eigenvalues of $\DtN_\Lambda(A(\ell))$ (corresponding to the radial bulk eigenfunctions), which are the two solutions $\sigma^{(-q^2)}_+\ge \sigma^{(-q^2)}_-$ of the quadratic equation 
\[
\left(\sigma - q\frac{I_1(q\ell)}{I_0(q\ell)}\right)\left(\sigma - q\frac{K_1(q)}{K_0(q)}\right) - \frac{K_0(q\ell)I_0(q)}{I_0(q\ell)K_0(q)}\left(\sigma + q\frac{K_1(q\ell)}{K_0(q\ell)}\right)\left(\sigma +q\frac{I_1(q)}{I_0(q)}\right)=0.
\]

The following properties of $\sigma_\pm^{(-q^2)}$ are easily deduced from the standard properties of the modified Bessel functions, see e.g. \cite[\S10]{dlmf}.
\begin{itemize}
\item The two branches never intersect, $\sigma^{(-q^2)}_+ > \sigma^{(-q^2)}_-$, and therefore represent two distinct real-analytic branches of eigenvalues.
\item As $q\to+\infty$, we have
\begin{equation}\label{eq:largealphaannulus} 
\sigma^{(-q^2)}_- = q - \frac{1}{2\ell}+o(1), \qquad \sigma^{(-q^2)}_+ = q + \frac{1}{2}+o(1).
\end{equation}
\item As $q\to+0$, we have
\begin{equation}\label{eq:smallalphaannulus} 
\sigma^{(-q^2)}_- \to 0:=\sigma^{(0)}_-,\qquad \sigma^{(-q^2)}_+ \to \frac{\ell+1}{\ell\log\ell}=:\sigma^{(0)}_+.
\end{equation}
The two limits coincide with the eigenvalues of $\DtN_0$ corresponding to the radial bulk eigenfunctions $U(r)=1$ and $U(r) = 1- \frac{(\ell+1)\log r}{\ell\log\ell}$, respectively.
\end{itemize}

Comparing \eqref{eq:largealphaannulus} and \eqref{eq:smallalphaannulus}, we conclude that for sufficiently large $q$,
\[
\sigma^{(-q^2)}_+ -  \sigma^{(0)}_+ - q =\frac{1}{2} - \frac{\ell+1}{\ell\log\ell} + o(1),
\]
with the positive leading term if $\ell>\ell_*$,  which completes the proof.

\begin{figure}[htb]
\centering
\includegraphics{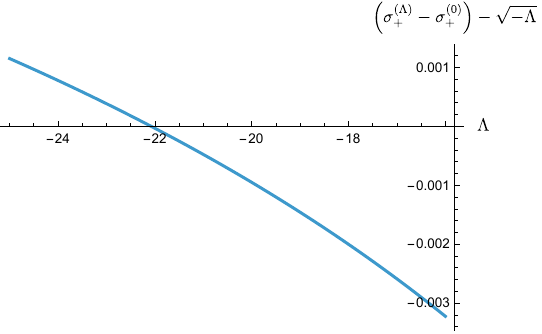}
\caption{The quantity  $\sigma^{(\Lambda)}_+ -  \sigma^{(0)}_+ - \sqrt{-\Lambda}$ for the annulus $A\left(10\right)$. Note that it becomes positive for $\Lambda \lessapprox -22.1$.}\label{fig:annulus}
\end{figure}

\begin{remark}\label{rem:infty}
For the annuli $A(\ell)$ of Theorem~\ref{thm:acounter} one has $\mathcal K_{A(\ell)}=1$, and by \eqref{eq:largealphaannulus}--\eqref{eq:smallalphaannulus} the upper radial branch satisfies $\sigma^{(\Lambda)}_+-\sigma^{(0)}_+\to\sqrt{-\Lambda}+\frac12$ in the iterated limit $\Lambda\to-\infty$, $\ell\to+\infty$.

At the same time, the subspace $\mathcal{U}\subset H^{1/2}(\partial A(\ell))$ of rotationally invariant boundary data is two-dimensional and invariant under $\DtN_\Lambda$, and $\sigma^{(\Lambda)}_\pm$ are precisely the ordered eigenvalues of the restriction $\DtN_\Lambda|_{\mathcal U}$.  As \eqref{eq:formHolderTheta} is an inequality between quadratic forms valid for \emph{all} boundary data, it can be restricted to $\mathcal U$, and the min--max principle applied within $\mathcal U$ gives, with  $\Lambda_1=\Lambda$ and $\Lambda_2=0$,
\[
\sigma^{(\Lambda)}_\pm-\sigma^{(0)}_\pm\le \Theta_{A(\ell)}(-\Lambda).
\]
Comparison with the displayed asymptotics shows that the additive constant $\frac{\mathcal{K}_\Omega}{2}$ in \eqref{eq:Kinf} cannot be improved, cf. Remark \ref{rem:compare}.
\end{remark}

\subsection{Proof of Theorem \ref{thm:bcounter}}

We start with two elementary estimates used in the proof of Theorem \ref{thm:bcounter}. 

For an annulus $A(\ell_1, \ell_2)$, $0<\ell_1<\ell_2$, let $-\Delta^{\Rob,\gamma_1,\gamma_2}_{A(\ell_1, \ell_2)}$ denote the Laplacian with boundary conditions $\partial_n U + \gamma_j U = 0$ on $C(\ell_j)$, $j=1,2$, given more explicitly in polar coordinates $(r,\theta)$ by $-\partial_r U + \gamma_1 U = 0$ at $r=\ell_1$, and $\partial_r U + \gamma_2 U = 0$ at $r=\ell_2$. Note that choosing $\gamma_j=0$ is equivalent to imposing the Neumann condition on the corresponding boundary.

We start with an estimate of the principal Robin eigenvalue for an annulus with a Neumann condition on the outer boundary.

\begin{lemma}\label{lem:bcounter-inner}
Fix $\ell>1$ and $\vartheta\in\left(0, \frac12\right)$. Assume that  $\gamma<0$. Then the principal eigenvalue 
\[
\lambda_1(\gamma):=\lambda_1\left(-\Delta^{\Rob,\gamma-\vartheta,0}_{A(1, \ell)}\right)
\]
of the mixed Robin/Neumann Laplacian in $A(1,\ell)$ satisfies
\[
\lambda_1\left(\gamma\right)>-\gamma^2
\]
for all sufficiently large $|\gamma|$.
\end{lemma}

\begin{proof}
The principal eigenvalue is negative and a corresponding eigenfunction is radial. Separating the variables and using the asymptotics of the modified Bessel functions, similarly to the proof of Theorem \ref{thm:acounter},  we get the asymptotics
\[
\sqrt{-\lambda_1(\gamma)} = -\gamma+\vartheta-\frac12+O\left(|\gamma|^{-1}\right)\qquad\text{as }\gamma\to -\infty.
\]
As $\vartheta<\frac12$, we have $\sqrt{-\lambda_1(\gamma)}<-\gamma$ for all sufficiently
large $|\gamma|$, and the result follows.
\end{proof}

\begin{lemma}\label{lem:bcounter-outer}
For $\tau>0$, consider a mixed Neumann/Robin Laplacian $-\Delta_{A(3/4,1)}^{\Rob, 0,-\tau}$ in the outer annulus $A(3/4,1)$. There exists a constant $P>0$ such that the number of its non-positive eigenvalues satisfies
\begin{equation}\label{eq:bcounter-outer-count}
\mathcal{N}^{-\Delta_{A(3/4,1)}^{\Rob, 0,-\tau}}(0)
\le P(\tau+1)
\end{equation}
\end{lemma}

\begin{proof}
We consider the mixed Steklov/Neumann spectral problem in the annulus $A(3/4, 1)$, with the Steklov condition on $C(1)$ and the Neumann condition on $C(3/4)$, denoting the corresponding Dirichlet-to-Neumann operator on $C(1)$ by $\tilde{\DtN}_0$. By the mixed problem analogue of \eqref{eq:DtNRdualityN}, we have
\[
\mathcal{N}^{-\Delta_{A(3/4,1)}^{\Rob, 0,-\tau}}(0) = \N^{\tilde{\DtN}_0}(\tau),
\] 
with the latter counting function behaving as $c\tau+O(1)$ as $\tau\to+\infty$ by Weyl's law; therefore, 
\eqref{eq:bcounter-outer-count} holds with $P=c+1$ for $\tau$ exceeding some $\tau_0>0$. Setting 
\[
P=\max\left\{c+1, \max_{\tau\in[0,\tau_0]}\frac{\N^{\tilde{\DtN}_0}(\tau)}{\tau+1}\right\}
\]
completes the proof. 
\end{proof}

\begin{proof}[Proof of Theorem \ref{thm:bcounter}]
Set
\begin{equation}\label{eq:zetaetc}
\zeta:=\er^{-1}, \qquad \ell:=\frac{\epsilon\zeta}{r_\epsilon}=\frac{\zeta}{\delta}=\er^3, \qquad\vartheta:=\frac25,
\end{equation}
see also \eqref{eq:delta}.
Thus
\[
\frac{1}{\log\ell} = \frac13 < \vartheta < \frac12.
\]
Choose a constant $q_0>0$ so large that Lemma
\ref{lem:bcounter-inner} applies with $\gamma=\gamma_0=-q_0$, and the values of $\ell$ and
$\vartheta$ given above, and define
\begin{equation}\label{eq:bcounter-parameters}
q_\epsilon:=\frac{q_0}{r_\epsilon}, \qquad
\beta_\epsilon:=\frac{\vartheta}{r_\epsilon}, \qquad
\eta_\epsilon:=q_\epsilon+\beta_\epsilon = \frac{q_0+\vartheta}{r_\epsilon}=\frac{-\gamma_0+\vartheta}{r_\epsilon}.
\end{equation}

We first prove that
\begin{equation}\label{eq:bcounter-zero-upper}
\sigma_{m_\epsilon}^{(0)}(\Omega_\epsilon) \le \beta_\epsilon.
\end{equation}
For $j=1,\dots,m_\epsilon$, define $U_{j,\epsilon}\in H^1(\Omega_\epsilon)$ by
\[
U_{j,\epsilon}(x) :=
\begin{cases}
\displaystyle
\frac{\log\left(\frac{\epsilon\zeta}{|x-x_j|}\right)}{\log\ell},
\qquad&\text{if } r_\epsilon<|x-x_j|<\epsilon\zeta,\\
0,\qquad&\text{otherwise}.
\end{cases}
\]
Because the disks $B(x_j,\epsilon)$ are pairwise disjoint and $\zeta<1$, the supports of these functions are pairwise disjoint.
Moreover,
\[
U_{j,\epsilon}(x)=1, \qquad x\in C(x_j,r_\epsilon),
\]
and its trace vanishes on every other component of
$\partial\Omega_\epsilon$. A direct calculation in polar coordinates
gives
\[
\|\nabla U_{j,\epsilon}\|_{L^2(\Omega_\epsilon)}^2 = \frac{2\pi}{\log\ell},
\qquad
\left\| U_{j,\epsilon}|_{\partial\Omega_\epsilon} \right\|_{L^2(\partial\Omega_\epsilon)}^2 = 2\pi r_\epsilon.
\]
Therefore, by
\eqref{eq:robinform},
\[
\mathfrak r_{-\beta_\epsilon}[U_{j,\epsilon}] = \frac{2\pi}{\log\ell} - 2\pi\beta_\epsilon r_\epsilon
= 2\pi \left(\frac1{\log\ell}-\vartheta\right) < 0.
\]
The same strict inequality holds for every non-zero element of
\[
\mathcal L_\epsilon :=
\operatorname{span}
\{U_{1,\epsilon},\dots,U_{m_\epsilon,\epsilon}\},
\qquad \dim\mathcal L_\epsilon=m_\epsilon,
\]
because the supports are pairwise disjoint. By  Glazman's Lemma,
\[
\mathcal N^{-\Delta_{\Omega_\epsilon}^{\Rob,-\beta_\epsilon}}(0) \ge m_\epsilon.
\]
Using the Robin--Dirichlet-to-Neumann counting identity \eqref{eq:DtNRdualityN} with $\Lambda=0$ and
$\sigma=\beta_\epsilon$, we obtain
\[
\mathcal N^{\DtN_0}(\beta_\epsilon) \ge m_\epsilon,
\]
which is equivalent to \eqref{eq:bcounter-zero-upper}.

We next show that, for all sufficiently small $\epsilon$,
\begin{equation}\label{eq:bcounter-negative-lower}
\sigma_{m_\epsilon}^{(-q_\epsilon^2)}(\Omega_\epsilon) > \eta_\epsilon.
\end{equation}
In order to do that, consider the Robin Laplacian $-\Delta_{\Omega_\epsilon}^{\Rob,-\eta_\epsilon}$.  We introduce artificial boundaries  
\[
\Gamma_\epsilon:= \bigcup_{j=1}^{m_\epsilon} C(x_j,\epsilon\zeta)\cup C\left(\frac34\right),
\]
which decompose $\Omega_\epsilon\setminus\Gamma_\epsilon$ as
\[
\Omega_\epsilon\setminus\Gamma_\epsilon = \bigsqcup_{j=1}^{m_\epsilon} A(x_j,\epsilon\delta,\epsilon\zeta)\sqcup A\left(\frac34,1\right) \sqcup \tilde{\Omega}_\epsilon, 
\]
where $\tilde{\Omega}_\epsilon$ is a domain with $\partial\tilde{\Omega}_\epsilon = \Gamma_\epsilon$; clearly
\[
\partial\left(\Omega_\epsilon\setminus\Gamma_\epsilon\right)=\partial\Omega_\epsilon \sqcup \Gamma_\epsilon.
\]

We consider the mixed Robin--Neumann spectral problem for the Laplacian in $\Omega_\epsilon\setminus\Gamma_\epsilon$ by keeping the Robin boundary conditions with parameter $-\eta_\epsilon$ on the original boundaries $\partial\Omega_\epsilon$ and imposing the Neumann conditions on the additional boundaries $\Gamma_\epsilon$; denote the resulting operator by 
\[
-\Delta^{\Rob, -\eta_\epsilon, 0}_{\Omega_\epsilon\setminus\Gamma_\epsilon}.
\]
The standard Neumann bracketing estimate gives
\begin{equation}\label{eq:bcounter-bracketing}
\mathcal N^{-\Delta_{\Omega_\epsilon}^{\Rob,-\eta_\epsilon}} \left(-q_\epsilon^2\right)
\le \mathcal N^{-\Delta^{\Rob, -\eta_\epsilon, 0}_{\Omega_\epsilon\setminus\Gamma_\epsilon}}\left(-q_\epsilon^2\right).
\end{equation}
The operator $-\Delta^{\Rob, -\eta_\epsilon, 0}_{\Omega_\epsilon\setminus\Gamma_\epsilon}$ is the direct sum of the resulting operators on separated pieces,
\[
-\Delta^{\Rob, -\eta_\epsilon, 0}_{\Omega_\epsilon\setminus\Gamma_\epsilon} = \bigoplus_{j=1}^{m_\epsilon} \left(-\Delta_{A(x_j, \epsilon\delta,\epsilon\zeta)}^{\Rob, -\eta_\epsilon, 0}\right) \oplus \left(-\Delta_{A(3/4,1)}^{\Rob, 0, -\eta_\epsilon}\right)\oplus \left(-\Delta_{\tilde{\Omega}_\epsilon}^{\Neu}\right),
\]
and its counting function is therefore the sum of individual counting functions.

There are three types of pieces. Trivially, as the Neumann Laplacian is non-negative,
\[
\mathcal N^{-\Delta_{\tilde{\Omega}_\epsilon}}\left(-q_\epsilon^2\right) = 0.
\]

Secondly, in order to deal with each inner annulus $A(x_j, \epsilon\delta,\epsilon\zeta)$, we shift its center to the origin and rescale it by the factor $\frac{1}{\epsilon\delta}=\frac{1}{r_\epsilon}$, which transforms it into $A\left(1,\frac{\zeta}{\delta}\right)=A(1,\ell)$, see  \eqref{eq:zetaetc}. The Neumann condition is not affected by the rescaling, and the standard Robin rescaling law \cite[\S3.1.3]{LMP} implies, using \eqref{eq:bcounter-parameters},
\[
\lambda_1\left(-\Delta_{A(x_j, \epsilon\delta,\epsilon\zeta)}^{\Rob, -\eta_\epsilon, 0}\right) =\frac{1}{r_\epsilon^2}\lambda_1\left(-\Delta_{A(1,\ell)}^{\Rob, -\eta_\epsilon r_\epsilon, 0}\right)
=\frac{1}{r_\epsilon^2}\lambda_1\left(-\Delta_{A(1,\ell)}^{\Rob, \gamma_0-\vartheta, 0}\right).
\]
Estimating the last expression using Lemma \ref{lem:bcounter-inner}, we get
\[
\lambda_1\left(-\Delta_{A(x_j, \epsilon\delta,\epsilon\zeta)}^{\Rob, -\eta_\epsilon, 0}\right) > -\frac{1}{r_\epsilon^2} \gamma_0^2 = -q_\epsilon^2,
\]
and therefore
\[
\mathcal N^{-\Delta_{A(x_j, \epsilon\delta,\epsilon\zeta)}^{\Rob, -\eta_\epsilon, 0}}
\left(-q_\epsilon^2\right) = 0, \qquad
j=1,\dots,m_\epsilon.
\]

Thirdly, by Lemma \ref{lem:bcounter-outer}, 
\begin{equation}\label{eq:bcounter-outer-bound}
\mathcal N^{-\Delta_{A(3/4,1)}^{\Rob, 0, -\eta_\epsilon}} \left(-q_\epsilon^2\right)
\le \mathcal N^{-\Delta_{A(3/4,1)}^{\Rob, 0, -\eta_\epsilon}}(0)
\le P\left(\eta_\epsilon+1\right) = P\left(\frac{-\gamma_0+\vartheta}{\epsilon\delta}+1\right)
\le \frac{P_1}{\epsilon}
\end{equation}
for all sufficiently small $\epsilon$, with a constant $P_1>0$ independent of $\epsilon$.

Combining \eqref{eq:bcounter-bracketing}--\eqref{eq:bcounter-outer-bound}, we conclude that
\[
\mathcal N^{-\Delta_{\Omega_\epsilon}^{\Rob,-\eta_\epsilon}}
\left(-q_\epsilon^2\right)
\le
\frac{P_1}{\epsilon}.
\]
On the other hand, by \eqref{eq:bcounter-packing}, $m_\epsilon\ge\frac{1}{16\epsilon^2}$, 
and consequently, for all sufficiently small $\epsilon$,
\[
\mathcal N^{-\Delta_{\Omega_\epsilon}^{\Rob,-\eta_\epsilon}} \left(-q_\epsilon^2\right) < m_\epsilon.
\]
Applying \eqref{eq:DtNRdualityN} with $\Lambda=-q_\epsilon^2$ and $\sigma=\eta_\epsilon$, we obtain
\[
\mathcal N^{\DtN_{-q_\epsilon^2}}(\eta_\epsilon) < m_\epsilon,
\]
which proves \eqref{eq:bcounter-negative-lower}.

Combining \eqref{eq:bcounter-zero-upper}, \eqref{eq:bcounter-negative-lower}, and \eqref{eq:bcounter-parameters}, we arrive at
\[
\sigma_{m_\epsilon}^{(-q_\epsilon^2)}(\Omega_\epsilon) - \sigma_{m_\epsilon}^{(0)}(\Omega_\epsilon) > \eta_\epsilon-\beta_\epsilon= q_\epsilon,
\]
completing the proof of  Theorem \ref{thm:bcounter}.
\end{proof}

\section{Generalisations and extensions}\label{sec:extensions}
\subsection{Consequence for Robin eigenvalues}

Fix $k\in\mathbb N$, and let the number $\gamma_1$, $\gamma_2$ satisfy
\begin{equation}\label{eq:robincondition}
\gamma_1\le \gamma_2\le-\sigma_k^{(0)}(\Omega).
\end{equation}
Then, by the Robin--Dirichlet-to-Neumann duality, the numbers $\Lambda_j:=\lambda_k^{\Rob,\gamma_j}(\Omega)$,
$j=1,2$, satisfy $\Lambda_1\le\Lambda_2\le0$, and are in the range of applicability  of Theorems~\ref{thm:convex},
\ref{thm:C2}, and \ref{thm:main}. 

\begin{prop}\label{prop:robin}
Let $\Omega\subset\R^d$ be a bounded domain, let $k\in\mathbb N$, and let
$\gamma_1\le\gamma_2$ satisfy \eqref{eq:robincondition}. If $\Omega$ has $C^2$ boundary, then
\[
\lambda_k^{\Rob,\gamma_2}(\Omega)-\lambda_k^{\Rob,\gamma_1}(\Omega) \ge \max\left\{
\frac{\left(\gamma_2-\gamma_1\right)^2}{\mathcal{C}_\Omega^{\,2}},
\left(\gamma_2-\gamma_1\right)^2-\mathcal{K}_\Omega\left(\gamma_2-\gamma_1\right)\right\},
\]
with $\mathcal{C}_\Omega$ and $\mathcal{K}_\Omega$ as in \eqref{eq:COmega} and \eqref{eq:KOmega}.
If $\Omega$ is convex, or has $C^2$ weakly mean-convex boundary, then
\[
\lambda_k^{\Rob,\gamma_2}(\Omega)-\lambda_k^{\Rob,\gamma_1}(\Omega) \ge  \left(\gamma_2-\gamma_1\right)^2.
\]
\end{prop}

\begin{proof}
With $\Lambda_j$ as above, \eqref{eq:DtNRdualityN} and the equality of multiplicities in
Proposition~\ref{prop:DtNRduality} give $\sigma_k^{(\Lambda_j)}(\Omega)=-\gamma_j$. 
Thus $\sigma_k^{(\Lambda_1)}-\sigma_k^{(\Lambda_2)}=\gamma_2-\gamma_1\ge0$, and the three
bounds follow from Corollaries~\ref{cor:C2-eigenvalues}, \ref{cor:main-eigenvalues} and
\ref{cor:convex-eigenvalues} respectively, in the last case after inverting the increasing
function $\Theta_\Omega$ of \eqref{eq:Theta}. \end{proof}

\begin{remark}
The two $C^2$ bounds are
complementary in the same way as Theorems~\ref{thm:C2} and \ref{thm:main}, cf.\
Remark~\ref{rem:compare}: only the first is nontrivial for $\gamma_2-\gamma_1<\mathcal{K}_\Omega$,
whereas the second retains the leading coefficient one, and that coefficient cannot be improved.
Indeed, for a bounded domain with smooth boundary and fixed $k$ one has
$\lambda_k^{\Rob,\gamma}(\Omega)=-\gamma^2+O\left(|\gamma|\right)$ as $\gamma\to-\infty$
\cite{LevitinParnovski, PankrashkinPopoff, HelfferKachmar}, 
so that for fixed $\gamma_2$ both
sides of the second bound of Proposition~\ref{prop:robin} are $\gamma_1^2(1+o(1))$.
\end{remark}

\subsection{Dirichlet-to-Neumann maps on Riemannian manifolds}

All the preceding arguments have a natural Riemannian version.  Let $(X,g)$ be a compact connected smooth Riemannian manifold of dimension $d\geq2$ with non-empty smooth boundary, write $X^\circ:=X\setminus\partial X$, and put
\[
\rho(x):=\dist_g(x,\partial X).
\]
For $\Lambda\leq0$ and $u\in H^{1/2}(\partial X)$, let $\mathcal E^g_\Lambda u\in H^1(X)$ be the solution of
\[
-\Delta_g U-\Lambda U=0\quad\text{in }X,\qquad U|_{\partial X}=u,
\]
and let $\DtN_\Lambda$ denote the corresponding Dirichlet-to-Neumann map.  Its quadratic form on $H^{1/2}(\partial X)$ is, similarly to \eqref{eq:quadDtN},
\[
\mathfrak d_{\Lambda}[u]
:=\int_X \left(|\nabla \mathcal E^g_\Lambda u|_g^2-\Lambda|\mathcal E^g_\Lambda u|^2\right)\,\dr\mathrm{v}_g.
\]

By analogy with \eqref{eq:MOmega}, set
\begin{equation}\label{eq:riemannian-lower}
\mathcal{M}_{g}:=\inf\left\{m\ge0:\ -\Delta_g\rho\geq-m\,\dr \mathrm{v}_g\ \text{ in }\mathscr D'(X^\circ)\right\},
\end{equation}
so that $-\Delta_g\rho\geq-\mathcal{M}_{g}\,\dr \mathrm{v}_g$, and set
\[
\mathcal{A}_{g}:=\sup_{0\neq u\in H^{1/2}(\partial X)}
\frac{\|\mathcal E^g_0u\|^2_{L^2(X)}}{\|u\|^2_{L^2(\partial X)}}.
\]  
The finiteness of $\mathcal{A}_{g}$ follows from the same  
argument as used in  the Euclidean case above. 

\begin{theorem}\label{thm:riemannian-general}
Let $(X,g)$ be as above and let $\Lambda_1\leq\Lambda_2\leq0$. Then
\begin{equation}\label{eq:riemannian-form}
0\leq \DtN_{\Lambda_1}-\DtN_{\Lambda_2} \leq \mathcal{C}_{g}\sqrt{\Lambda_2-\Lambda_1}\,I
\end{equation}
in the sense of quadratic forms on $H^{1/2}(\partial X)$, where one may take
\[
\mathcal{C}_{g}:=1+\mathcal{M}_{g}\mathcal{A}_{g}.
\]
\end{theorem}

\begin{proof}
Only the notation changes from the Euclidean proof.  Put $\mu_\rho=-\Delta_g\rho$ and, for $\gamma\in\mathbb R$, set
\[
\mathfrak r^g_\gamma[U] :=\int_X|\nabla U|_g^2\,\dr \mathrm{v}_g +\gamma\int_{\partial X}|U|^2\,\dr S_g,
\qquad U\in H^1(X).
\]
One has $|\nabla\rho|_g=1$ almost everywhere, and the smooth boundary collar gives the boundary term $\partial_n\rho=-1$.  The cutoff argument used in the $C^2$ Euclidean case yields the Riemannian analogue of \eqref{eq:Greenrho}.  Hence, for $V=\er^{-\alpha\rho}U$, the calculation in Proposition~\ref{prop:sub} gives
\[
\mathfrak r^g_{\gamma-\alpha}[V]+\alpha^2\|V\|^2_{L^2(X)} \leq \mathfrak r^g_\gamma[U]+\alpha \mathcal{M}_{g}\|V\|^2_{L^2(X)}.
\]
The Dirichlet principle then gives \eqref{eq:riemannian-form} exactly as in \eqref{eq:formgeneral-intermediate}.  The uniform estimate for the $L^2(X)$ norm of the $\Lambda_2$-harmonic extension follows from the Riemannian Dirichlet Laplacian by the same resolvent argument as in \eqref{eq:L2extension}.
\end{proof}

\begin{remark}\label{rem:riemannian-theta}
The other choice of $\alpha$ made in \S\ref{sec:proofconvex} is available verbatim in this setting as well: taking $\alpha$ to be the non-negative root of $\alpha^2-\mathcal{M}_{M,g}\alpha=\Lambda_2-\Lambda_1$ gives the Riemannian analogue of Theorem~\ref{thm:main},
\[
0\leq \DtN_{\Lambda_1}-\DtN_{\Lambda_2} \leq \frac{\mathcal{M}_{g}+\sqrt{\mathcal{M}_{g}^2+4\left(\Lambda_2-\Lambda_1\right)}}{2}\,I
\leq \left(\sqrt{\Lambda_2-\Lambda_1}+\mathcal{M}_{g}\right)I,
\]
in which the constant $\mathcal{A}_{g}$ does not appear.
\end{remark}

\begin{corollary}\label{cor:riemannian-eigenvalues}
Under the assumptions of Theorem~\ref{thm:riemannian-general}, if $\sigma_k^{(\Lambda)}(X,g)$ denote the ordered Dirichlet-to-Neumann eigenvalues, then
\[
0\leq \sigma_k^{(\Lambda_1)}(X,g)-\sigma_k^{(\Lambda_2)}(X,g) \leq \mathcal{C}_{g}\sqrt{\Lambda_2-\Lambda_1},\qquad k\in\mathbb N.
\]
\end{corollary}

The sharp constant one is recovered from standard geometric assumptions.

\begin{corollary}\label{cor:riemannian-ricci}
Assume, in addition, that
\[
\operatorname{Ric}_g\geq0\quad\text{in }X^\circ, \qquad
\mathcal{H}_{\partial X}\geq0,
\]
where $\mathcal{H}_{\partial X}$ is the mean curvature with respect to the outward unit normal. Then one may take $\mathcal{M}_{g}=0$ in Theorem~\ref{thm:riemannian-general}. In particular,
\[
0\leq \DtN_{\Lambda_1}-\DtN_{\Lambda_2} \leq\sqrt{\Lambda_2-\Lambda_1}\,I
\]
in the sense of quadratic forms, and
\[
0\leq \sigma_k^{(\Lambda_1)}(X,g)-\sigma_k^{(\Lambda_2)}(X,g) \leq\sqrt{\Lambda_2-\Lambda_1}.
\]
\end{corollary}

\begin{proof}
Under these assumptions the boundary distance is superharmonic in the distributional sense,
\[
-\Delta_g\rho\geq0.
\]
This follows from \cite[Corollary~(2.44)]{Kasue}. Thus $\mathcal{M}_{g}=0$ in \eqref{eq:riemannian-lower}, and the result follows from Theorem~\ref{thm:riemannian-general}.
\end{proof}

\begin{remark}
Theorem~\ref{thm:riemannian-general} isolates the actual analytic input: a distributional lower bound for $-\Delta_g\rho$.  Non-negative Ricci curvature and weak mean convexity are a convenient sufficient condition giving the optimal value $\mathcal{M}_{g}=0$, but other geometric hypotheses leading to a  distributional inequality for the Laplacian of the distance function may be used in exactly the same way.
\end{remark}

\subsection{Dirichlet-to-Neumann eigenvalues on  metric graphs}

The preceding argument has a simple analogue for compact metric graphs, which makes the role of the distance function particularly transparent.  Let $\mathcal G$ be a finite connected metric graph, let $V_\Dir$ be a non-empty set of boundary vertices, and impose the standard Kirchhoff conditions at the non-empty set $V_\Neu$ of all remaining vertices. 
For simplicity, assume that every vertex in $V_\Dir$ has degree one.  For $\Lambda\le 0$, let $\DtN^{\mathcal G}_{\Lambda}$ denote the Dirichlet-to-Neumann map corresponding to
\[
-U'' - \Lambda U=0
\]
on the edges, with Dirichlet data prescribed on $V_\Dir$, and write
\[
\sigma_1^{(\Lambda)}(\mathcal G)\leq\cdots\leq \sigma_{|V_\Dir|}^{(\Lambda)}(\mathcal G)
\]
for its eigenvalues.  We refer to \cite{BerKu}, in particular its \S3.5, for more details of spectral theory of metric (quantum) graphs and corresponding Dirichlet-to-Neumann maps.

The natural analogue of Theorem~\ref{thm:convex} for graphs is the quadratic-form comparison
\[
0\leq \DtN^{\mathcal G}_{\Lambda_1}-\DtN^{\mathcal G}_{\Lambda_2} \leq \sqrt{\Lambda_2-\Lambda_1}\,I,
\qquad \Lambda_1\leq\Lambda_2\leq 0.
\]
By the min--max principle, this implies the eigenvalue inequality
\begin{equation}\label{eq:graph-conjB}
0\leq \sigma_k^{(\Lambda_1)}(\mathcal G)-\sigma_k^{(\Lambda_2)}(\mathcal G)
\leq \sqrt{\Lambda_2-\Lambda_1},\qquad k=1,\dots,|V_\Dir|, \qquad \Lambda_1\le\Lambda_2\le 0.
\end{equation}

Set
\[
\rho(x):=\dist(x,V_\Dir).
\]
On each edge, $\rho$ is a concave piecewise affine function with $|\rho'|=1$ almost everywhere; its only possible interior singularity is a downward corner.  Branching, however, introduces an additional vertex contribution.  If $v\in V_\Neu$, write
\[
\Sigma_\rho(v):=\sum_{e\sim v}\partial_e\rho(v),
\]
where $\partial_e$ denotes the derivative away from $v$ along $e$.  If $m_v$ is the number of incident edges which start a shortest path from $v$ to $V_\Dir$, then
\[
\Sigma_\rho(v)=\deg(v)-2m_v.
\]
Thus $\Sigma_\rho(v)\leq0$ precisely when at least half of the edges incident to $v$ point ``downhill'' towards the boundary.  In distributional terms, on the interior of the graph,
\[
\rho''=-2\sum_{x_*}\delta_{x_*}+\sum_{v\in V_\Neu}\Sigma_\rho(v)\,\delta_v,
\]
where $x_*$ runs over the interior downward corners of $\rho$.  Consequently,
\begin{equation}\label{eq:graph-superharmonic}
\Sigma_\rho(v)\leq0\quad\text{for every }v\in V_\Neu
\end{equation}
is exactly the metric-graph counterpart of the superharmonicity condition $-\Delta\rho\geq0$ used in \eqref{eq:convexdistmain}.  Notice that concavity along each individual edge is automatic; the only obstruction comes from the vertex masses created by branching. We stress that \eqref{eq:graph-superharmonic} is a sufficient condition for \eqref{eq:graph-conjB}, not a necessary one.

Under \eqref{eq:graph-superharmonic}, the substitution and quadratic-form argument of \S\ref{sec:proofconvex} carry over to metric graphs almost verbatim.  Indeed, the same substitution \eqref{eq:VU}, $x\in\mathcal{G}$, gives the graph analogue of Proposition~\ref{prop:sub}, and the vertex terms have the correct sign exactly because of \eqref{eq:graph-superharmonic}; the Dirichlet principle is unchanged.  Hence the quadratic-form comparison above, and therefore \eqref{eq:graph-conjB}, hold whenever the boundary vertices are leaves and $\rho$ is superharmonic in the above sense.  This includes, for example, subdivided intervals (more generally, the condition is automatic at every interior vertex of degree at most two), as well as equilateral rooted trees whose Dirichlet leaves all lie at the same depth; in particular it holds for equal-length stars.  The relevant notion here is therefore not geodesic convexity of a subset of a graph, but rather superharmonicity of the distance to the boundary, i.e. a discrete analogue of mean convexity.

The condition \eqref{eq:graph-superharmonic} can fail even for a tree, and so can \eqref{eq:graph-conjB}.  Consider the star graph with three Dirichlet leaves and one Kirchhoff centre, with edge lengths
\[
0.1,\ 5,\ \text{and }10.
\]
At the centre there is a unique edge which starts a shortest path to $V_\Dir$, so that $m_v=1$ and
$\Sigma_\rho(v)=3-2=1>0$.  A direct computation of the boundary Dirichlet-to-Neumann matrix gives the plots shown in Figure \ref{fig:tree} thus producing a counterexample to graph analogues of both Conjectures \ref{conj:A} and \ref{conj:B}.

\begin{figure}[htb]
\centering
\includegraphics{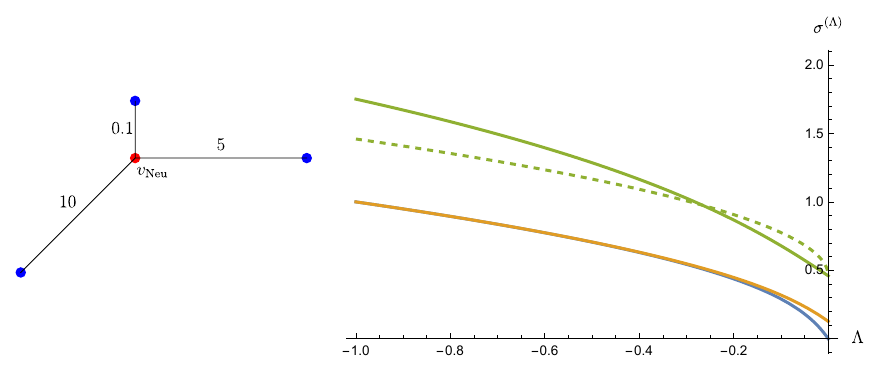}
\caption{On the left, a metric graph (not to scale) with one Neumann vertex and three Dirichlet leaves. On the right, its three Dirichlet-to-Neumann eigenvalues plotted against $\Lambda$ (solid lines) and the curve $\sigma_3^{(0)}+\sqrt{-\Lambda}$ (dashed line).}\label{fig:tree}
\end{figure}

More generally, without assumptions on the degrees of Kirchhoff vertices, the same argument as in the Euclidean case produces the bound 
\[
\sigma_k^{(\Lambda_1)}(\mathcal G) - \sigma_k^{(\Lambda_2)}(\mathcal G) \le  \sqrt{\Lambda_2-\Lambda_1} +\frac{1}{2\er} \sum_{v\in V_\Neu}\frac{\left(\Sigma_\rho(v)\right)_+}{\rho(v)},
\]
for $\Lambda_1\le \Lambda_2\le 0$.

\subsection{Wiggly domains}\label{sec:wiggly} 

The constants $\mathcal K_\Omega$ and $\mathcal C_\Omega$ of Theorems~\ref{thm:C2} and
\ref{thm:main} may degenerate as one leaves the class $C^{1,1}$, even along
real-analytic simply connected domains converging to the disk.  Indeed, given
$\kappa\in(0,1)$ and $p\in\mathbb N$, let $\mathcal W_{p,\kappa}\subset\R^2$ be the domain
bounded by the curve
\begin{equation}\label{eq:wiggly}
\theta\ \longmapsto\ \er^{\ir\theta}+p^{-1-\kappa}\er^{\ir p\theta},
\qquad \theta\in[0,2\pi),
\end{equation}
where we identified $\mathbb{R}^2$ with $\mathbb{C}$, see  Figure \ref{fig:wiggly}. 

\begin{figure}[htb]
\centering
\includegraphics{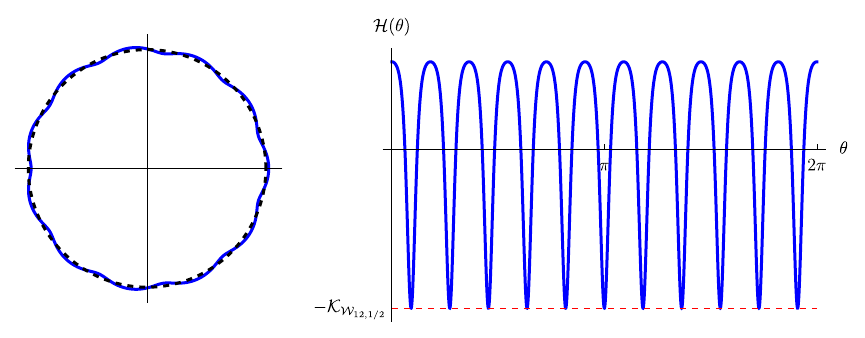}
\caption{On the left, the boundary of the wiggly domain $\mathcal W_{12,1/2}$, see \eqref{eq:wiggly}, together with the unit circle $\mathbb S$ (dashed); the Hausdorff distance between them is $12^{-3/2}\approx0.024$. On the right, the curvature of $\partial\mathcal W_{12,1/2}$, oscillating between $\frac{12}{\sqrt{12}+1}\approx2.68$ and $-\mathcal{K}_{\mathcal W_{12,1/2}}=-\frac{12}{\sqrt{12}-1}\approx-4.87$ (dashed line).}\label{fig:wiggly}
\end{figure}

Elementary explicit computations, which we omit, show that for $p$ large enough
(depending on $\kappa$) this is a Jordan curve; that $\mathcal W_{p,\kappa}\to\mathbb D$ as
$p\to\infty$ in the Hausdorff distance at the rate $p^{-1-\kappa}$ and
$\partial\mathcal W_{p,\kappa}\to\mathbb S$ in $C^{1,\beta}$ for every $\beta<\kappa$, the family
being bounded in $C^{1,\kappa}$ and unbounded in $C^{1,\beta}$ for every $\beta>\kappa$;
and that the curvature of $\partial\mathcal W_{p,\kappa}$ oscillates at the scale $p^{-1}$
with amplitude of order $p^{1-\kappa}$, so that
\[
\mathcal K_{\mathcal W_{p,\kappa}}=p^{1-\kappa}\left(1+O\!\left(p^{-\kappa}\right)\right)
\qquad
\mathcal C_{\mathcal W_{p,\kappa}} \ge  \frac12\,p^{1-\kappa}\left(1+o(1)\right).
\]
Note that since $\sigma_k^{(\Lambda)}\!\left(\mathcal W_{p,\kappa}\right)\to\sigma_k^{(\Lambda)}(\mathbb D)$
for fixed $k$ and $\Lambda$, this is not a counterexample to Conjectures~\ref{conj:A}
and~\ref{conj:B}.

\section*{Acknowledgements}
\phantomsection
\addcontentsline{toc}{section}{Acknowledgements}
The authors would like to thank David Sher, Marco Marletta, Konstantin Pankrashkin,  Gregory Berkolaiko, and Jussi Behrndt for valuable discussions. 

M. Levitin was partially supported by EPSRC grant no. EP/V051881/1 and a Simons Fellowship.

K.-M. Perfekt was supported by grant no. 334466 of the Research Council of Norway,
``Fourier Methods and Multiplicative Analysis.'' 

I. Polterovich was partially supported by supported by NSERC and FRQNT.

The authors are grateful to the Isaac Newton Institute for Mathematical Sciences, Cambridge, for support and hospitality during the programme ``Geometric Spectral Theory and Applications'', where work on this paper was undertaken. This work was supported by EPSRC grant EP/Z000580/1. 

\section*{AI usage disclosure}
\phantomsection
\addcontentsline{toc}{section}{AI usage disclosure}
The authors used different large language models, in particular ChatGPT 5.6 and Claude Opus 5,  extensively for mathematical discussions and editorial assistance. In particular, AI suggested that the geometric properties of the distance function to the boundary should play a key role  in the argument, and  helped constructing  counterexamples. All mathematical proofs, computations, and references were subsequently checked, revised, and written in their final form by the authors.

{\small 

\end{document}